\documentclass[10pt,a4paper]{article}

\usepackage[utf8]{inputenc}

\usepackage[american]{babel}
\usepackage[
			hmargin=1.95cm,
			vmargin=1.95cm,
			twoside,
			a4paper,			
			nomarginpar]
			{geometry}
\usepackage{amssymb,amsthm}
\usepackage[fleqn]{amsmath}
\usepackage{xspace} 
\usepackage[dvipsnames]{xcolor}
\usepackage{todonotes}
\usepackage{dsfont}
\usepackage{microtype}
\usepackage{appendix}
\DisableLigatures[f]{encoding = *}

\newcommand{\myTitle}{Tunneling behavior of Ising and Potts models on grid graphs\xspace}
\newcommand{\myName}{Alessandro Zocca\xspace}

\definecolor{dred}{RGB}{220,0,0}
\definecolor{halfgray}{gray}{0.55}
\definecolor{webgreen}{rgb}{0,.5,0}
\definecolor{webbrown}{rgb}{.6,0,0}
\usepackage[hyperfootnotes=false,pdfpagelabels]{hyperref}
\hypersetup{%
    colorlinks=true, linktocpage=true, pdfborder={0 0 0}, pdfstartpage=3, pdfstartview=FitV,%
    breaklinks=true, pdfpagemode=UseNone, pageanchor=true, pdfpagemode=UseOutlines,%
    plainpages=false, bookmarksnumbered, bookmarksopen=true, bookmarksopenlevel=1,%
    hypertexnames=true, pdfhighlight=/O,
    urlcolor=webbrown, linkcolor=RoyalBlue, citecolor=webgreen, 
    pdftitle={\myTitle},%
    pdfauthor={\textcopyright\ \myName},%
}  

\usepackage{tabularx} 
\usepackage{caption}
\usepackage{subfig}  
\usepackage{placeins}
\usepackage{psfrag}
\usepackage{graphicx}
\usepackage{epstopdf}

\usepackage{bm}
\usepackage{soul}

\usepackage{listings} 
\newcommand{\ie}{i.e., }

\newcommand{\ed}{\,{\buildrel d \over =}\,}
\newcommand{\cd}{\xrightarrow{d}}
\newcommand{\E}{\mathbb E}

\newcommand{\pr}[1]{\mathbb P \Big ( #1 \Big )}
\newcommand{\prin}[1]{\mathbb P ( #1 )}
\newcommand{\sut}{\mathbin{\lvert}}

\newcommand{\N}{\mathbb N}
\newcommand{\R}{\mathbb R}
\newcommand{\Z}{\mathbb Z}

\renewcommand{\b}{\beta}

\newcommand{\e}{\epsilon}
\newcommand{\h}{\eta}
\renewcommand{\o}{\omega}

\renewcommand{\l}{\lambda}
\newcommand{\s}{\sigma}

\newcommand{\G}{\Gamma}

\renewcommand{\L}{\Lambda}

\newcommand{\cX}{\mathcal{X}}

\renewcommand{\DH}{\Delta H}
\newcommand{\tising}{\tau^{\bm{-1}}_{\bm{+1}}}

\renewcommand{\ss}{\cX^s}

\newcommand{\binf}{\b \to \infty}
\newcommand{\limb}{\lim_{\binf}}

\newcommand{\rmexp}{\mathrm{Exp}}

\newcommand{\ck}{\mathbf{s}_k}
\newcommand{\cc}{\mathbf{s}} 
\newcommand{\dd}{\mathbf{s}'} 

\newcommand{\tcd}{\tau^{\cc}_{\dd}}

\newcommand{\xtbb}{\{X_{t}^\b\}_{t \in \N}}

\newtheorem{thm}{Theorem}[section]
\newtheorem{cor}[thm]{Corollary}
\newtheorem{lem}[thm]{Lemma}
\newtheorem{prop}[thm]{Proposition}

\title{Tunneling behavior of Ising and Potts models\\in the low-temperature regime} 
\date{\today}
\author{F.R.~Nardi \and A.~Zocca}

\begin{document}
\maketitle
\begin{abstract}
We consider the ferromagnetic $q$-state Potts model with zero external field in a finite volume and assume that the stochastic evolution of this system is described by a Glauber-type dynamics parametrized by the inverse temperature $\b$. Our analysis concerns the low-temperature regime $\binf$, in which this multi-spin system has $q$ stable equilibria, corresponding to the configurations where all spins are equal. Focusing on grid graphs with various boundary conditions, we study the tunneling phenomena of the $q$-state Potts model. More specifically, we describe the asymptotic behavior of the first hitting times between stable equilibria as $\binf$ in probability, in expectation, and in distribution and obtain tight bounds on the mixing time as side-result. In the special case $q=2$, our results characterize the tunneling behavior of the Ising model on grid graphs.
\end{abstract}

\section{Introduction and main results}
\label{sec1}

\subsection{Model description}
\label{sub11}
The Potts model is a canonical statistical physics model born as a natural extension~\cite{P52} of the Ising model in which the number of possible local spins goes from two to a general integer number $q \in \N$. 

The $q$-state Potts model is a spin system characterized by a set $S=\{1,\dots,q\}$ of spins values and by a finite graph $G = (V,E)$, which describes the spatial structure of the finite volume where the spins interact. A configuration $\s \in S^V$ assigns a spin value $\s(v) \in S$ to each vertex $v \in V$. We denote by $\cX=S^{V}$ the set of all possible spin configurations on the graph $G$. The edge set $E$ of the graph $G$ describes the pairs of vertices whose spins interact with each other. The \textit{Hamiltonian} or \textit{energy function} $H: \cX \to \R$ associates an energy with each configuration $\s \in \cX$ according to
\begin{equation}
\label{eq:energyfunction_Potts}
	H(\s) := - J_c \sum_{(v,w) \in E} \mathds{1}_{\{\s(v) = \s(w)\}}, \quad \s \in \cX,
\end{equation}
where  $J_c$ is the \textit{coupling or interaction constant}. Such an energy function corresponds to the situation where there is no external magnetic field and, in fact, $H(\s)$ describes only the local interactions between nearest-neighbor spins present in configuration $\s$. The \textit{Gibbs measure} for the $q$-state Potts model on $G$ is the probability distribution on $\cX$ defined by
\begin{equation}
\label{eq:gibbs}
	\mu_\b(\s):=\frac{e^{-\b H(\s)}}{\sum_{\s' \in \cX} e^{-\b H(\s')}}, \quad \s \in \cX,
\end{equation}
where $\b >0$ is the \textit{inverse temperature}.

The Potts model is called \textit{ferromagnetic} when $J_c >0$ and \textit{antiferromagnetic} when $J_c <0$. In the ferromagnetic case, the Gibbs measure $\mu_\b$ favors configurations where neighboring spins have the same value. On the contrary, in the antiferromagnetic case, neighboring spins are more likely not to be aligned. In this paper we focus on the ferromagnetic Potts model and, without loss of generality, we take $J_c=1$, since in absence of a magnetic field it amounts to rescaling of the temperature. 

We assume the spin system to evolve according to a stochastic Glauber-type dynamics described by a single-spin update Markov chain $\smash{\xtbb}$ on $\cX$ with transition probabilities between any pair of configurations $\s,\s' \in \cX$ given by
\begin{equation}
\label{eq:metropolistransitionprobabilities_Potts}
	P_\b(\s,\s'):=
	\begin{cases}
		Q(\s,\s') e^{-\b [H(\s')-H(\s)]^+}, 	& \text{ if } \s \neq \s',\\
		1-\sum_{\h \neq \s} P_\b(\s,\h), 	& \text{ if } \s=\s',
	\end{cases}
\end{equation}
where $Q$ is the \textit{connectivity matrix} that allows only single-spin updates, \ie for every $\s,\s' \in \cX$ we set
\begin{equation}
\label{eq:connectivityfunction_Potts}
	Q(\s,\s'):=
	\begin{cases}
		\frac{1}{q |V|}, 								& \text{if } \left |\{v \in V : \s(v)\neq \s'(v)\} \right |=1,\\
		0, 													& \text{if } \left |\{v \in V : \s(v)\neq \s'(v)\} \right |>1.
	\end{cases}
\end{equation}
The matrix $Q$ is clearly symmetric and irreducible, and the resulting dynamics $P_\b$ is reversible with respect to the Gibbs measure $\mu_\b$ given in~\eqref{eq:gibbs}. One usually refers to the triplet $(\cX, H, Q)$ as \textit{energy landscape} and to~\eqref{eq:metropolistransitionprobabilities_Potts} as \textit{Metropolis transition probabilities}.

The considered Metropolis dynamics can be described in words as follows. At each step a vertex $v \in V$ and a spin value $k \in S$ are selected independently and uniformly at random and the current configuration $\s \in \cX$ is updated in vertex $v$ to spin $k$ with a probability that depends only on the neighboring spins of $v$. More specifically, denote by $\s^{v,k}$ the configuration obtained from $\s$ by changing the spin of vertex $v$ into $k$ and calculate the energy difference
\[
	H(\s^{v,k}) - H(\s)= \sum_{w \sim v} \mathds{1}_{\{\s(w) = \s(v) \}} -  \mathds{1}_{\{ \s(w) = k \}} ,
\]
and the spin of vertex $v$ is updated to $k$ with probability $1$ if $H(\s^{v,k}) - H(\s) \leq 0$ or with probability $e^{-\b (H(\s^{v,k}) - H(\s))}$ if $H(\s^{v,k}) - H(\s) >0$. Two examples of this single-spin update dynamics are showed in Figures~\ref{fig:transition1} and~\ref{fig:transition2}. In both examples we consider the Potts model with $q=4$ and display only a single vertex $v$ and its neighbors. The four different spins are displayed using different colors via the mapping $\{1,2,3,4\} \longleftrightarrow \{\tikz\draw[fill=white] (0,0) circle (.75ex); , \, \tikz\draw[fill=gray!45!white] (0,0) circle (.75ex); , \, \tikz\draw[fill=gray] (0,0) circle (.75ex); , \, \tikz\draw[fill=black!75!gray] (0,0) circle (.75ex); \}$. In each example, we start from an initial configuration $\s$ (Figures~\ref{fig:transition1}(a) and~\ref{fig:transition2}(a)) and illustrate all the possible non-trivial transitions obtained by updating the spin in vertex $v$ to each of the other $q-1$ different spins.
\begin{figure}[!h]
	\centering
	\subfloat[Initial configuration $\s$]{\includegraphics{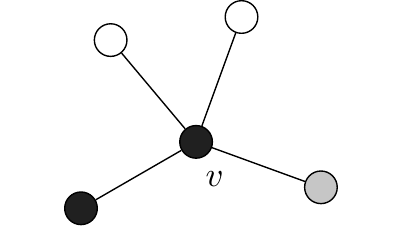}}
	\hspace{0.25cm}
	\subfloat[$H(\s^{v,1}) - H(\s) =-1$]{\includegraphics{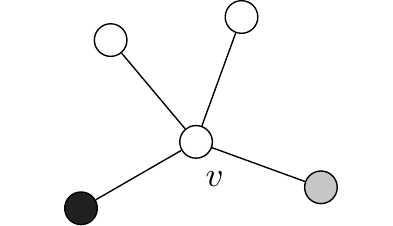}}
	\hspace{0.25cm}
	\subfloat[$H(\s^{v,2}) - H(\s) =0$]{\includegraphics{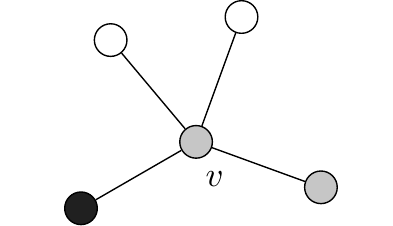}}
	\hspace{0.25cm}
	\subfloat[$H(\s^{v,3}) - H(\s) =1$]{\includegraphics{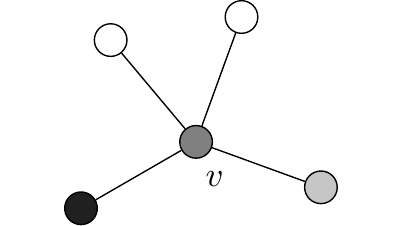}}
	\caption{List of possible non-trivial transitions (with the corresponding energy difference) from the initial configuration $\s$ with $\s(v)=4$ in (a) when the spin in vertex $v$ is updated.}
	\label{fig:transition1}
\end{figure}
\vspace{-0.5cm}
\begin{figure}[!h]
	\centering
	\subfloat[Initial configuration $\s$]{\includegraphics{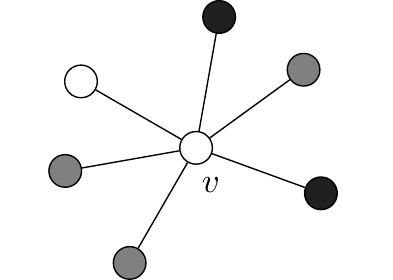}}
	\hspace{0.2cm}
	\subfloat[$H(\s^{v,2}) - H(\s) =1$]{\includegraphics{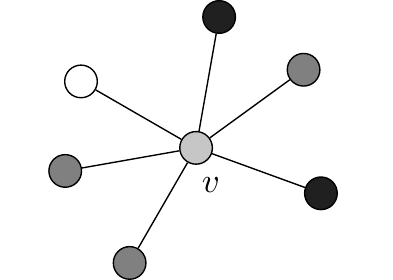}}
	\hspace{0.2cm}
	\subfloat[$H(\s^{v,3}) - H(\s)=-2$]{\includegraphics{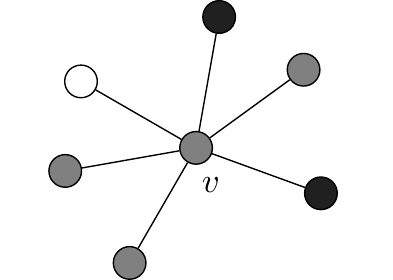}}
	\hspace{0.2cm}
	\subfloat[$H(\s^{v,4}) - H(\s) =-1$]{\includegraphics{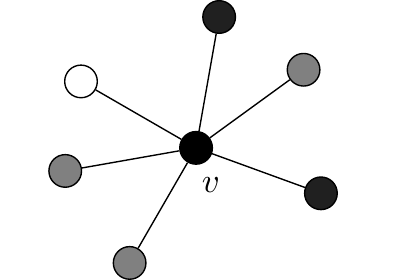}}
	\caption{List of possible non-trivial transitions (with the corresponding energy difference) from the initial configuration $\s$ with $\s(v)=1$ in (a) when the spin in vertex $v$ is updated.}
	\label{fig:transition2}
\end{figure}

\subsection{Main results}
\label{sub12}
In the present paper we focus on the analysis of the $q$-state ferromagnetic Potts model in the low-temperature regime $\binf$, where the spin system is in the so-called \textit{ordered phase} the coexistence of multiple equilibrium states. Indeed, in this regime the stationary distribution $\mu_\b$ concentrates around the global minima of the Hamiltonian $H$, which in the case of a connected graph $G$, are exactly $q$ and correspond to the configurations where all the spins have the same value. We denote them by $\cc_1,\dots,\cc_q$, with the convention that $\ck \in \cX$ is the configuration where all the spins are equal to $k$, namely $\ck(v) = k$ for every $v \in V$. In the following we will refer to them as the \textit{stable configurations} of the interacting spin system and denote their collection as $\ss$.

In the low-temperature regime these stable configurations and their basins of attraction become \textit{traps}, in the sense that the Markov chain $\smash{\xtbb}$ cannot move quickly between them. Intuitively, along any possible trajectory the Markov chain $\smash{\xtbb}$ must visit mixed-spin configurations that are highly unlikely in view of~\eqref{eq:gibbs} and the time to reach such configurations is correspondingly long. Due to these exponentially long transition times between stable configurations, the considered dynamics exhibits the so-called \textit{slow} or \textit{torpid mixing}. We characterize the low-temperature behavior of the $q$-state Potts model in terms of first hitting times and mixing times, which we will now introduce. 

For a nonempty subset $A \subset \cX$ and a configuration $\s \in \cX \setminus A$, we denote by $\tau^\s_A$ the \textit{first hitting time} of the subset $A$ for the Markov chain $\smash{\xtbb}$ with initial configuration $\s$ at time $t=0$, \ie
\[
	\tau^\s_A:=\inf \{ t >0 : X^{\b}_t \in A \sut X_0^\b=\s\}.
\]
The hitting time $\tau^\s_A$ is often called \textit{tunneling time} when both the starting and target configurations are stable configurations, \ie $\{\s\} \cup A \subseteq \ss$. For every $0 < \e < 1$ define the \textit{mixing time}  $t^{\mathrm{mix}}_\b(\e)$ of the Markov chain $\smash{\xtbb}$ as
\[
	t^{\mathrm{mix}}_\b(\e):=\min\{ n \geq 0 ~:~ \max_{x \in \cX} \| P^n_\b(x,\cdot) - \mu_\b(\cdot) \|_{\mathrm{TV}} \leq \e \},
\]
where $\| \nu - \nu' \|_{\mathrm{TV}}:=\frac{1}{2} \sum_{x \in \cX} |\nu(x)-\nu'(x)|$ denotes the total variation distance between two probability distributions $\nu,\nu'$ on $\cX$. The mixing time $t^{\mathrm{mix}}_\b(\e)$ describes the rate of convergence of the Markov chain $\smash{\xtbb}$ to its stationary distribution $\mu_\b$ and is intimately related to the  \textit{spectral gap} of the Markov chain, which is defined in terms of the eigenvalues $1=\l_{\b}^{(1)} > \l_{\b}^{(2)} \geq \dots \geq \l_\b^{(|\cX|)} \geq -1$ of the transition matrix $(P_\b(\s,\s'))_{\s,\s' \in \cX}$ as $\smash{\rho_\b := 1-\l_{\b}^{(2)}}$.

Our analysis focuses in the present paper on the dynamics of the $q$-state Potts model on finite two-dimensional rectangular lattices, to which we will simply refer to as \textit{grid graphs}. More precisely, given two integers $K,L \geq 2$, we will take the graph $G$ to be a $K \times L$ grid graph $\L$ with two possible boundary conditions: periodic and open.

The main result of this paper concerns the asymptotic behavior of the tunneling times between stable configurations: for any pair of stable configuratiosn $\cc, \dd$, we give asymptotic bounds in probability for $\tau^{\cc}_{\ss \setminus \{\cc\}}$ and $\tcd$, identify the order of magnitude of their expected values and prove that their asymptotic rescaled distribution is exponential. We further identify the precise exponent at which the mixing time of the Markov chain $\xtbb$ asymptotically grows as $\b$ and show that it depends up to a constant factor on the smaller side length of $\L$. 

\begin{thm}[Low-temperature behavior of the Potts model on grid graphs] \label{thm:main}
Consider the Metropolis Markov chain $\xtbb$ corresponding to the $q$-state Potts model on the $K \times L$ grid $\L$ with $\max \{K,L\} \geq 3$. Let $\G(\L)>0$ be the constant defined as
\begin{equation}
\label{eq:gammaL}
	\G(\L):=
	\begin{cases}
		2 \min\{K,L\} + 2 & \text{ if } \L \text{ has periodic boundary conditions},\\
		\min\{K,L\} +1 & \text{ if } \L \text{ has open boundary conditions}.
	\end{cases}
\end{equation}
Then, for any $\cc,\dd \in \ss$, $\cc \neq \dd$, the following statements hold:
\begin{itemize}
	\item[\textup{(i)}] For every $\e >0$ $\displaystyle \limb \pr{ e^{\b( \G(\L) - \e)} < \tau^{\cc}_{\ss \setminus \{\cc\}} \leq \tcd < e^{\b ( \G(\L) +\e)}} =1$;
	\item[\textup{(ii)}] $\displaystyle \limb \frac{1}{\b} \log \E \tcd =  \limb \frac{1}{\b} \log \E \tau^{\cc}_{\ss \setminus \{\cc\}} = \G(\L)$;
	\item[\textup{(iii)}] $ \displaystyle \frac{\tau^{\cc}_{\ss \setminus \{\cc\}}}{\E \tau^{\cc}_{\ss \setminus \{\cc\}}} \cd \rmexp(1), \quad \textup{ as } \binf$;
	\item[\textup{(iv)}] $ \displaystyle \frac{\tcd}{\E \tcd} \cd \rmexp(1), \quad \textup{ as } \binf$;
	\item[\textup{(v)}] For any $\e \in (0,1)$ $\displaystyle \limb \b^{-1} \log t^{\mathrm{mix}}_\b(\e) = \G(\L)$ and there exist two constants $0 < c_1 \leq c_2 < \infty$ independent of $\b$ such that
\begin{equation}
\label{eq:rho}
	\forall \, \b>0 \qquad c_1 e^{-\b \G(\L)} \leq \rho_\b \leq c_2 e^{-\b \G(\L)}.
\end{equation}
\end{itemize}
\end{thm}
We remark that in the low-temperature limit the total number $q$ of possible spin values does not appear in our main result because we focus on \textit{logarithmic} equivalences and the number $q$ does not affect the \textit{order of magnitude} of the tunneling times and neither that of the mixing time. This is the case also for analogous results for mixing times of heat-bath and Swenden-Wang dynamics derived in~\cite{Borgs2012}, for which the dependence on the grid side length is the same. The bounds in~\cite{Borgs2012} are valid for a more general $d$-dimensional grid, while ours are specialized for the case $d=2$, for which we obtain sharper exponents.

From our analysis it is easy to derive analogous results for a $K \times L$ grid graph $\L$ with semi-periodic boundary conditions (\ie periodic on the horizontal boundaries and open on the vertical ones), in which the the value that the exponent $\G(\L)$ would be $\min\{K+2,2L+1\}$. We expect that analogous results hold also for rectangular regions $\L$ of other lattices (e.g.~triangular, hexagonal, Kagome lattices) with an exponent $\G(\L)$ that would depend, up to a constant, on the minimum side length.\\

In the particular case in which there are only $q=2$ spin values, the Potts model reduces to Ising model with no external magnetic field, which has exactly two stable configurations that we denote as $\ss=\{\bm{-1},\bm{+1}\}$, see Figure~\ref{fig:colorconventionsising} for an illustration. In the following corollary we rewrite our main result for the tunneling time for the Ising model.

\begin{cor}[Low-temperature behavior of the Ising model on grid graphs] \label{thm:main_ising}
Consider the Metropolis Markov chain $\xtbb$ corresponding to the Ising model on the $K \times L$ grid $\L$ with $\max \{K,L\} \geq 3$ and define the constant $\G(\L)>0$ as in~\eqref{eq:gammaL}. Then
\begin{itemize}
	\item[\textup{(i)}] For every $\e >0$ $\displaystyle \limb \pr{ e^{\b( \G(\L) - \e)} < \tising < e^{\b ( \G(\L) +\e)}} =1$;
	\item[\textup{(ii)}] $\displaystyle \limb \frac{1}{\b} \log \E \tising =  \G(\L)$;
	\item[\textup{(iii)}] $ \displaystyle \frac{\tising}{\E \tising} \cd \rmexp(1), \quad \textup{ as } \binf$;
	\item[\textup{(iv)}] For any $\e \in (0,1)$ $\displaystyle \limb \b^{-1} \log t^{\mathrm{mix}}_\b(\e) = \G(\L)$ and there exist two constants $0 < c_1 \leq c_2 < \infty$ independent of $\b$ such that
\begin{equation}
\label{eq:rhoIsing}
	\forall \, \b>0 \qquad c_1 e^{-\b \G(\L)} \leq \rho_\b \leq c_2 e^{-\b \G(\L)}.
\end{equation}
\end{itemize}
\end{cor}

Similar results for the hitting times of the Ising model have already been proved in~\cite{Thomas1989}. More precisely, the following lower bound for the expected hitting times of a certain subset of states $S_{\L}$ on the $d$-dimensional cube $\L \subset \Z^d$ of side $L$:
\[
	\max_{\s \in \cX} E \tau^\s_{S_\L} \geq c(\b) e^{\beta \alpha^* L^{d-1}},
\]
where $\alpha^*$ is a constant independent of $\L$, $\beta$ and $c(\beta)>0$ is a function that does not depend on $\L$. Since $\bm{+1} \in S_{\L}$, our result can be seen as a refinement of~\cite[Proposition 2.3]{Thomas1989} in dimension $d=2$, as (a) we identify the precise constant $\alpha^*$ showing how it depends on the type of boundary conditions, (b) we indirectly prove that $\limb \frac{1}{\b} \log c(\beta) = 0$, and (c) we derive a matching upper bound.

Furthermore, statement (iv) in Corollary~\ref{thm:main_ising} improves the estimates on the spectral gap presented in~\cite[Proposition 2.5]{Thomas1989}, since as illustrated in~\eqref{eq:rhoIsing} also for this quantity we identify the exact exponent and obtain a matching upper bound. In the special case of open boundary conditions, our result for the spectral gap should be compared with the estimates given in~\cite[Theorem 4.1]{Martinelli1994} (valid for more general dynamics) and the asymptotics for $L \to \infty$ proved in~\cite[Theorem 1.4]{Cesi1996}. For related results concerning the equilibrium properties of the Ising model on finite lattices with zero magnetic field see also~\cite{Chayes1987}.

\subsection{Related results and discussion}
\label{sub13}
The Potts model is one of the most studied statistical physics models and is named after Renfrey Potts, who introduced the model in his Ph.D.~thesis~\cite{P52} in 1951. The model was related to the ``clock model'' or  ``planar Potts'', a variant of which was introduced earlier in~\cite{AT43} and is known as the \textit{Ashkin-Teller model}

The Potts model has been studied so extensively both by mathematicians and physicists that an exhaustive review of the related literature would be very long and out of the scope of this paper. Nevertheless, we now outline some related work that focus on the equilibrium or dynamical properties of the Potts model that are most relevant for this paper.

The equilibrium properties of the Potts model, such as the phase transition, critical temperature, and their dependence on $q$, have been studied on various infinite graphs, such as the square lattice $\mathbb Z^d$~\cite{Baxter1973,Baxter1982b}, the triangular lattice~\cite{Baxter1978,EW82}, and the Bethe lattice~\cite{Ananikyan1995,Aguiar1991,DiLiberto1987}. If the underlying structure is described instead by a complete graph, then we obtain the mean-field version of Potts model, also known as \textit{Curie-Weiss Potts model}, which received a lot of attention in the literature~\cite{Costeniuc2005,Ellis1990,Ellis1992,Gandolfo2010,Wang1994}.

Another branch of research focuses more on the dynamical properties of the Potts model, investigating in particular mixing times for various types of dynamics, the most studied ones being Glauber~\cite{BCP16,BCFKTVV99,Cuff2012,Galvin2012b,GKRS12,Galvin2007,Gheissari2016,Gheissari2016a,GJMP06,Jerrum1995}, Swendsen-Wang dynamics~\cite{BCFKTVV99,Borgs2012,Cooper2000,Cooper1999,GSV15,Gheissari2016,Gheissari2016b,Gheissari2017,Ullrich2013}. The focus of this part of literature is to describe at a given temperature how the mixing time grows as a function the graph size $n=|V|$ and the number of colors $q$. In particular, the goal is to distinguish whether the considered dynamics has fast or slow mixing depending on the type of the graph and its properties, such as boundary conditions or dimensions in the specific case of grid graphs.

Metastability is a dynamical property with a similar flavor as tunneling that has been studied for various spin systems. In particular, the metastability for the mean-field $3$-state Potts model with a non-reversible dynamics has been analyzed in~\cite{Landim2016} for fixed temperature in the thermodynamic limit. In this paper we focus instead on the $q$-state Potts model on finite volume with Metropolis dynamics, for which we obtained the results for tunneling outlined in Theorem~\ref{thm:main}.

There is an extensive literature concerning the metastable behavior for the Ising model on square lattices with Glauber dynamics, which relates with the results presented in Corollary~\ref{thm:main_ising}. More specifically, results in finite-volume case have been derived in~\cite{BL11,BAC96,BM02,CL98,KO93,KO94,NS91} and in the infinite-volume case in~\cite{BdHS10,CM13,Dehghanpour1997,SS98}. Results have been obtained for the metastability of the Ising model also on the hypercube~\cite{Jovanovski2017} and on certain types of random graphs~\cite{Dommers2016,Dommers2017}. Another related $3$-spin system which has been studied with similar techniques in~\cite{CN13,Cirillo1996,MO01} is the \textit{Blume-Capel model}. Tunneling phenomena have been studied also for other models with Metropolis dynamics, such as the \textit{hard-core model}~\cite{NZB16,Zocca2017} and the \textit{Widom-Rowlinson model}~\cite{Zocca2017b}.

In the present paper we study the low-temperature behavior of the Potts model using the \textit{pathwise approach} (see~\cite{OV05} for a sistematic overview and further references) and its more recent extensions~\cite{CNS14b,MNOS04,NZB16}, but also other techniques have been successfully used in the literature to study tunneling and metastability phenomena, e.g.~the \textit{potential theoretical approach} (introduced in~\cite{BEGK02}, for an overview see the recent book~\cite{Bovier2015a}) and the \textit{martingale approach}~\cite{BeltranLandim10,BL11,BL14}.

\section{Geometry of Potts configurations and energy landscape analysis}
\label{sec2}

This section is devoted to the analysis of some geometrical and combinatorial properties of the Potts configurations on grid graphs. This analysis will then be leveraged to prove some structural properties of the energy landscape $(\cX,H,Q)$ of the Potts model on grid graphs, which are presented in Theorem~\ref{thm:chadw}. These properties are precisely the model-dependent characteristics that are needed to exploit the general framework developed in~\cite{NZB16} to derive the main result presented in Subsection~\ref{sub12} for the asymptotic behavior of the Potts model in the low-temperature regime.

We first introduce some definition and notation that will be used in the rest of the paper. The connectivity matrix $Q$ given in~\eqref{eq:connectivityfunction_Potts} is irreducible, \ie for any pair of configurations $\s,\s'\in \cX$, $\s \neq \s'$, there exists a finite sequence $\o$ of configurations $\o_1,\dots,\o_n \in \cX$ such that $\o_1=\s$, $\o_n=\s'$ and $Q(\o_i,\o_{i+1})>0$, for $i=1,\dots, n-1$. We will refer to such a sequence as a \textit{path} from $\s$ to $\s'$ and we will denote it by $\o: \s \to \s'$. Given a path $\o=(\o_1,\dots,\o_n)$, we define its \textit{height} $\Phi_\o$ as
\begin{equation}
\label{eq:pathheight}
	\Phi_\o:= \max_{i=1,\dots,n} H(\o_i).
\end{equation}
The \textit{communication energy} between two configurations $\s,\s' \in\cX$ is the minimum value that has to be reached by the energy in every path $\o: \s \to \s'$, \ie
\begin{equation}
\label{eq:ch}
	\Phi(\s,\s') := \min_{\o : \s \to \s'} \Phi_\o = \min_{\o : \s \to \s'} \max_{\h \in\o} H(\h). 
\end{equation}
Given two nonempty disjoint subsets $A,B \subset \cX$, we define the communication energy between $A$ and $B$ by
\begin{equation}
\label{eq:chAB}
	\Phi(A,B) := \min_{\s \in A, \, \s' \in B} \Phi(\s,\s').
\end{equation}

\begin{thm}[Structural properties of energy lanscape] \label{thm:chadw}
Consider the energy landscape $(\cX, H, Q)$ corresponding to the Potts model on a $K \times L$ grid $\L$. Then:
\begin{itemize}
\item[\textup{(i)}] For every $\cc,\dd \in \ss$, $\cc \neq \dd$
\[
	 \Phi(\cc,\dd) -H(\cc)=  \G(\L) =
	 \begin{cases}
	 	2 \min\{K,L\}+2 		& \text{ if } \L \text{ has periodic boundary conditions},\\
	 	\min\{K,L\}+1 			& \text{ if } \L \text{ has open boundary conditions}.
	 \end{cases}
\]
\item[\textup{(ii)}] The following inequality holds:
\begin{equation}
\label{eq:adw}
	\Phi(\s,\ss) - H(\s) <  \G(\L) \qquad \forall \, \s \in \cX \setminus \ss.
\end{equation}
\end{itemize}
\end{thm}
Property~\eqref{eq:adw} is usually referred to as \textit{absence of deep wells}. Theorem~\ref{thm:chadw} explains why is precisely the constant $\Gamma(\L)$ defined in~\eqref{eq:gammaL} that appears in Theorem~\ref{thm:main} and Corollary~\ref{thm:main_ising}, characterizing the behavior of tunneling times and mixing times in the low-temperature regime. 

The rest of the section is organized as follows: in Subsection~\ref{sub21} we introduce some useful notation and definitions that will be used throughout the section, while in Subsection~\ref{sub22} we describe the geometric properties of Potts configurations that will be of interest for our analysis. Later, in Subsection~\ref{sub23} we present an \textit{expansion algorithm} for Potts configurations on grid graphs, that will be leveraged in Subsection~\ref{sub24} to build paths between stable configurations. Subsection~\ref{sub25} is devoted to the derivation of lower bounds for the communication height between stable configurations. Lastly, in Subsection~\ref{sub26} we present the proof of Theorem~\ref{thm:chadw}, that combines the expansion algorithm introduced in Subsection~\ref{sub23} and the inequalities derived in Subsections~\ref{sub24} and~\ref{sub25}.

\subsection{Definitions and notation}
\label{sub21}
In this subsection we introduce some notation and definitions tailored for the Potts model on a grid graph $\L$ (valid regardless of the chosen boundary conditions, unless specified otherwise) that will be used in the rest of the section.

A $K \times L$ grid graph $\L=(V,E)$ has vertex set $V=\{0,\dots,L-1\} \times \{0,\dots,K-1\}$ and every vertex $v \in \L$ is naturally identified by its coordinates $(v_1, v_2)$, where $v_1$ denotes the column and $v_2$ the row where $v$ lies. We denote by $c_j$, $j=0,\dots,L-1$, the $j$-th column of $\L$, \ie the collection of vertices whose horizontal coordinate is equal to $j$, and by $r_i$, $i=0,\dots,K-1$, the $i$-th row of $\L$, \ie the collection of vertices whose vertical coordinate is equal to $i$. With a minor abuse of notation, we will write $(v,w) \in r_i$ when $(v,w) \in E$ is a horizontal edge that links two vertices $v,w$ both on row $r_i$. Similarly $(v,w) \in c_j$ when  $(v,w) \in E$ is a vertical edge that links two vertices on column $c_j$.

It is convenient to visualize a $K\times L$ grid graph $\L$ by means of $K L$ squares, each of them corresponding to a vertex of the grid graph $\L$, as illustrated in Figure~\ref{fig:squares}. Note that this representation respects the adjacency relations: the neighbors of a given vertex $v$ are in one-to-one correspondence with the squares that share an edge with the square corresponding to $v$. In particular, this equivalent representation corresponds to the Peierls contours on the dual graph $\L+\left (\frac{1}{2},\frac{1}{2} \right )$.
\begin{figure}[!h]
	\centering
	\subfloat[Grid graph with periodic boundary conditions with highlighted two vertices, $v$ and $w$, and their corresponding neighbors]{\includegraphics{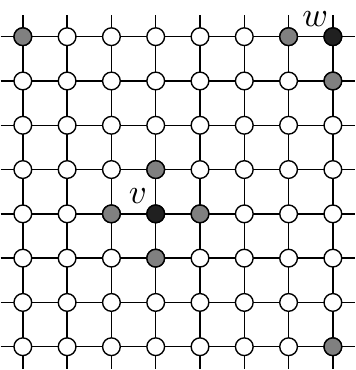}}
	\hspace{0.5cm}
	\subfloat[Equivalent representation of the grid graph with periodic boundary conditions with highlighted the squares corresponding to $v, w$ and their corresponding neighborhoods]{\includegraphics{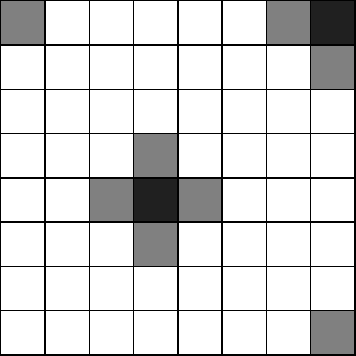}}
	\hspace{0.5cm}
	\subfloat[Grid graph with open boundary conditions with highlighted two vertices, $v$ and $w$, and their corresponding neighbors]{\includegraphics{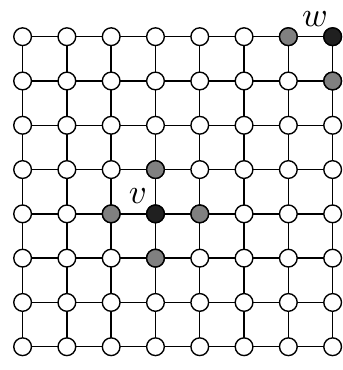}}
	\hspace{0.5cm}
	\subfloat[Equivalent representation of grid graph with open boundary conditions with highlighted the squares corresponding to $v, w$ and their corresponding neighborhoods]{\includegraphics{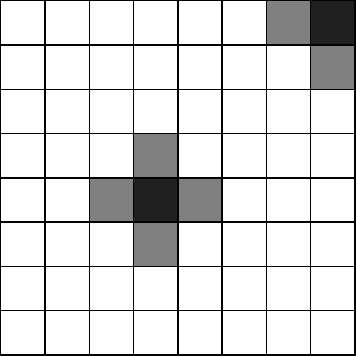}}
	\caption{Equivalent representation of a spin system on a $8\times 8$ grid graphs with different boundary conditions}
	\label{fig:squares}
\end{figure}
\FloatBarrier

For brevity, we will interchangeably refer to the spin value of a vertex using its color, since we can define a one-to-one mapping between spin values and colors as we did in Figures~\ref{fig:transition1} and~\ref{fig:transition2}. This convention and the equivalent representation with squares should help the reader to visualize $q$-state Potts configurations on a grid graph $\L$ as collection of clusters of $q$ different colors, see three examples in Figure~\ref{fig:colorconventions}.

\begin{figure}[!h]
	\centering
	\subfloat[Potts configuration with $q=2$ \label{fig:ccA}]{\includegraphics{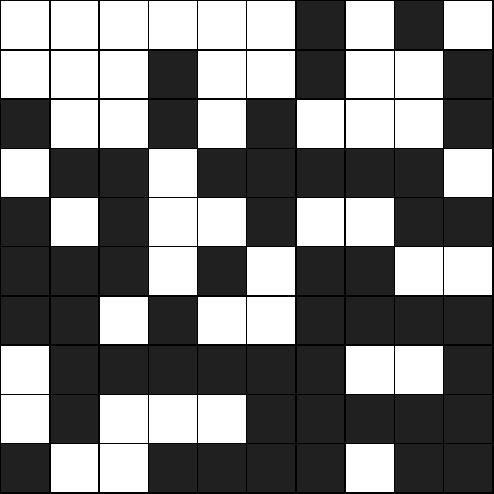}}
	\hspace{0.6cm}
	\subfloat[Potts configuration with $q=3$]{\includegraphics{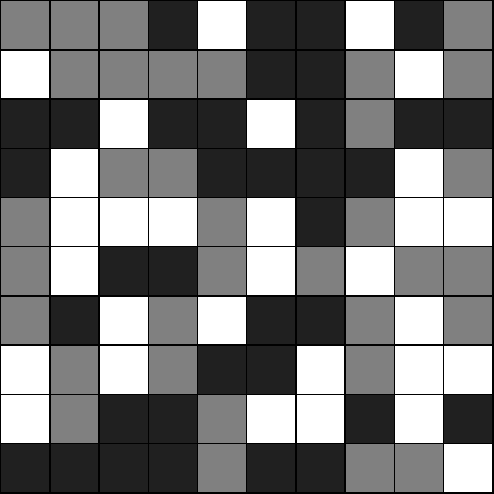}}
	\hspace{0.6cm}
	\subfloat[Potts configuration with $q=4$]{\includegraphics{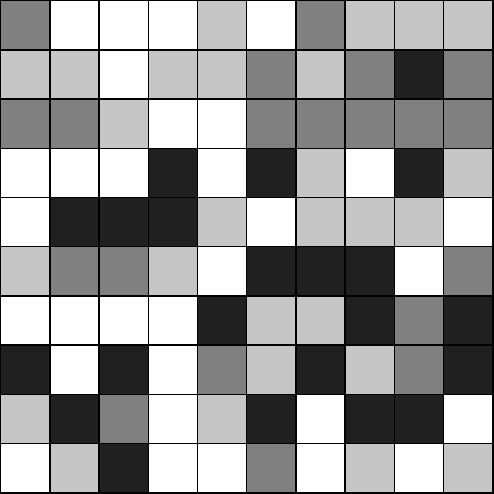}}
	\caption{Examples of Potts configuration on the $10 \times 10$ grid}
	\label{fig:colorconventions}
\end{figure}

\FloatBarrier
\newpage
Note that in the special case of $q=2$ spin values, the Potts model reduces to the classical Ising model, in which the two spin values are usually identified using the symbols $+$ and $-$, as illustrated in Figure~\ref{fig:colorconventionsising}. However, since in this paper we are interested in the case of a general $q \in \N$, we will use only the visualization using colors as in Figure~\ref{fig:colorconventions}.
\begin{figure}[!h]
	\centering
	\includegraphics{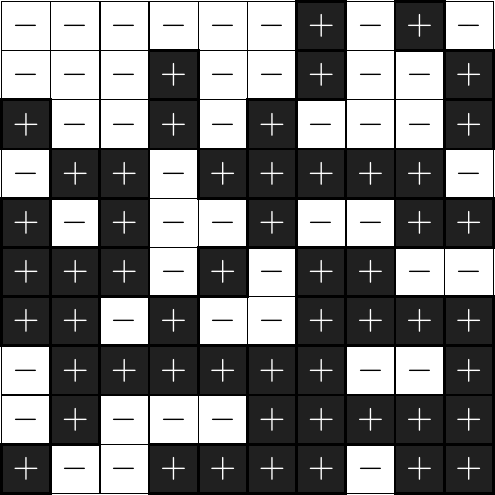}
	\caption{The Ising configuration on the $10 \times 10$ grid corresponding to that in Figure~\ref{fig:ccA}}
	\label{fig:colorconventionsising}
\end{figure}

Define the \textit{energy gap} $\DH(\s)$ of a configuration $\s \in \cX$ as the difference between its energy and the energy of any stable configuration, \ie 
\begin{equation}
\label{eq:def_DH}
	\DH(\s):=H(\s) - H(\cc), \quad \text{ for any } \cc \in \ss.
\end{equation}
Given a configuration $\s \in \cX$, we call an edge $e =(v,w)\in E$ \textit{disagreeing} if it connects two vertices with different colors, \ie $\s(v)\neq \s(w)$. From the definition of the energy function~\eqref{eq:energyfunction_Potts}, it follows that $\DH(\s)$ is equal to the number of disagreeing edges that configuration $\s$ has, since
\begin{equation}
\label{eq:Udisagreeingedges}
	\DH(\s)=H(\s) + |E| = |E| - \sum_{(v,w) \in E} \mathds{1}_{\{\s(v) =\s(w)\}} = \sum_{(v,w) \in E} \mathds{1}_{\{\s(v) \neq \s(w)\}}.
\end{equation}
Note that the energy gap $\DH(\s)$ corresponds to the total perimeter of the same-color clusters that configuration $\s$ has. Indeed, the sides of the squares used to represent Potts configuration, see e.g.~Figures~\ref{fig:colorconventions} and~\ref{fig:colorconventionsising}, are precisely the edges of the dual graph of $\L$ used to define the cluster contours. Hence, the energy gap $\DH(\s)$ quantifies the \textit{surface tension} between the clusters of different colors that configuration $\s$ has. Indeed, the disagreeing edges of a Potts configuration $\s$ on $\L$ are in one-to-one correspondence with the edges of the dual graph $\L+\left (\frac{1}{2},\frac{1}{2} \right )$ that define the Peierls contours of configuration $\s$.

We will now use in a crucial way the structure of the grid graph $\L$ to rewrite the energy gap. The edges of a grid $\L$ can have either vertical or horizontal orientation, and we can partition the edge set $E$ accordingly. More precisely, we consider the two subsets of vertical edges $E_v$ and horizontal edges $E_h$, which are such that $E = E_h \cup E_v$ and $E_h \cap E_v = \emptyset$. In view of this partition of the edge set $E$ and of the structure of the Potts energy function $H$ defined in~\eqref{eq:energyfunction_Potts}, we can rewrite the energy gap of a configuration $\s \in \cX$ as the sum of two contributions of horizontal and vertical edges, namely
\begin{equation}
\label{eq:DHhv}
	\DH(\s)=\sum_{(v,w) \in E_v} \mathds{1}_{\{\s(v) \neq \s(w)\}} + \sum_{(v,w) \in E_h} \mathds{1}_{\{\s(v) \neq \s(w)\}}.
\end{equation}
This identity is essentially saying that the total length of the Peierls contours of a given configuration can be seen as the sum of two terms, the total number of vertical segments and the total number of horizontal segments the contours consist of. In the rest of the paper, it will be convenient to have the following notation. Let $\DH_{r_i}(\s)$ be the energy gap of a configuration $\s \in \cX$ in the $i$-th row, namely
\begin{equation}\label{eq:Uri}
	\DH_{r_i}(\s) := \sum_{(v,w) \in r_i} \mathds{1}_{\{\s(v) \neq \s(w)\}}.
\end{equation}
Similarly, we define $\DH_{c_j}(\s)$ as the energy gap of a configuration $\s \in \cX$ in the $j$-th column, \ie
\begin{equation}\label{eq:Ucj}
	\DH_{c_j}(\s) := \sum_{(v,w) \in c_j} \mathds{1}_{\{\s(v) \neq \s(w)\}}.
\end{equation}
Note that the energy gap of a configuration $\s \in \cX$ on the horizontal (vertical) edges respectively can be rewritten as the sum of the energy gaps on each row (respectively, column), \ie
\begin{equation}\label{eq:UhUv}
	 \sum_{(v,w) \in E_h} \mathds{1}_{\{\s(v) \neq \s(w)\}} = \sum_{i=0}^{K-1} \DH_{r_i}(\s) \quad \text{ and } \quad  \sum_{(v,w) \in E_v} \mathds{1}_{\{\s(v) \neq \s(w)\}} = \sum_{j=0}^{L-1} \DH_{c_j}(\s).
\end{equation}

Given a Potts configuration $\s \in \cX$ on $\L$, a vertex $v \in \L$, and a color $k \in \{1,\dots,q\}$, we define $\s^{v,k} \in \cX$ to be the configuration obtained from $\s$ by coloring the vertex $v$ with color $k$, i.e.
\begin{equation}\label{eq:svk}
	\s^{v,k}(w):=
	\begin{cases}
		\s(w) 	& \text{ if } w\neq v,\\
		k			& \text{ if } w=v.
	\end{cases}
\end{equation}

\subsection{Local geometric properties: Bridges and crosses}
\label{sub22}
In this subsection we will introduce some geometric features of Potts configurations on a $K \times L$ grid graph $\L$ and study how they are related with their corresponding energy.

We say that a configuration $\s \in \cX$ has a \textit{horizontal bridge} on a row if all the vertices on that row have the same color. \textit{Vertical bridges} are defined analogously. A few examples of bridges are illustrated in Figure~\ref{fig:bridgescross}(a) and (b). The next lemma is an immediate consequence of the structure of rows and columns on $\L$.
\begin{lem}
\label{lem:monochromaticcross}
A Potts configuration on $\L$ cannot display simultaneously a horizontal bridge and a vertical bridge of different colors.
\end{lem}

A configuration $\s \in \cX$ is said to have a \textit{cross} when it has both a vertical and horizontal bridges. Note that, in view of Lemma~\ref{lem:monochromaticcross}, these two bridges cannot be of different colors. Figure~\ref{fig:bridgescross}(c) shows an example of a cross. If the specific color $k \in \{1,\dots,q\}$ of bridges (crosses) is relevant, we will refer to them as \textit{$k$-bridges} (\textit{$k$-crosses}) or specify their color.

\begin{figure}[!h]
	\centering
	\subfloat[A horizontal black bridge]{\includegraphics{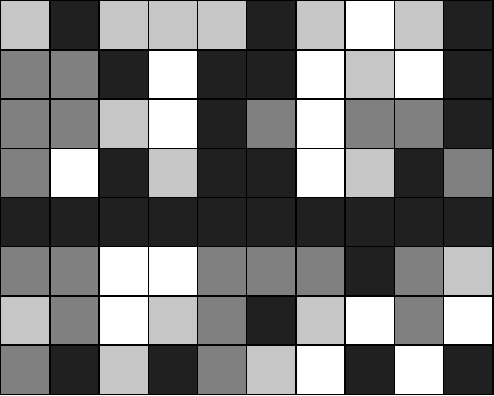}}
	\hspace{0.6cm}
	\subfloat[Two vertical bridges, one black and one white]{\includegraphics{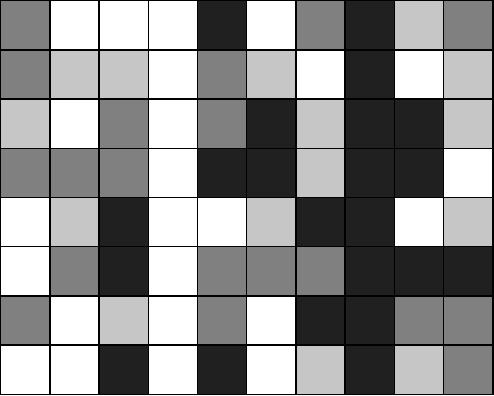}}
	\hspace{0.6cm}
	\subfloat[A black cross]{\includegraphics{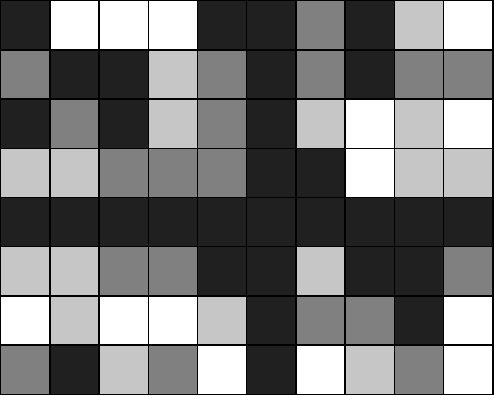}}
	\caption{Example of configurations on a $8 \times 10$ grid graph displaying black bridges or a black cross}
	\label{fig:bridgescross}
\end{figure}

\FloatBarrier

\begin{lem}[Bridges and zero energy gap rows/columns]
\label{lem:zerowastage}
The following properties hold for every Potts configuration $\s \in \cX$ on a grid graph $\L$:
\begin{itemize}
	\item[\textup{(a)}] $\DH_{r}(\s) = 0$ if and only if $\s$ has a horizontal bridge on row $r$;
	\item[\textup{(b)}] $\DH_{c}(\s) = 0$ if and only if $\s$ has a vertical bridge on column $c$.
\end{itemize}
Furthermore, if $\L$ is a grid graph with periodic boundary conditions, then
\begin{itemize}
	\item[\textup{(c)}] If $\s$ has no horizontal bridge on row $r$, then $\DH_{r}(\s)\geq 2$;
	\item[\textup{(d)}] If $\s$ has no vertical bridge on column $c$, then $\DH_{c}(\s) \geq 2$.
\end{itemize}
\end{lem}

Lemma~\ref{lem:zerowastage}(c) and (d) states that in the case of periodic boundary conditions if a configuration has no bridge on a given row/column, the surface tension on that row/column is at least $2$. 

\begin{proof}
Statements (a) and (b) are an immediate consequence of~\eqref{eq:Uri} and~\eqref{eq:Ucj}: indeed the energy gap on a row/column can be seen as the total number of disagreements on that row/column, which is equal to zero for horizontal/vertical bridges. 

We will prove only property (c), since the proof of (d) is analogous after interchanging the roles of rows and columns. Consider a $K \times L$ grid $\L$ with periodic boundary conditions. It follows from (a) that if $\s$ has no horizontal bridge on row $r$, then $\DH_r(\s) \geq 1$ and therefore, to prove statement (c), it is enough to show that $\DH_r(\s) \neq 1$. Suppose by contradiction that there exists a configuration $\s \in \cX$ and a row $r$ of $\L$ such that $\DH_r(\s)=1$. Let $v$ and $w$ the only two neighboring vertices such that $\s(v)\neq \s(w)$ and let $e=(v,w)$ be the edge that links them. Since $\DH_r(\s)=1$, $e$ is the unique disagreeing edge on row $r$, which means that the remaining $L-1$ edges create a path from $v$ to $w$ where all the comprised vertices must be of the same color and thus, in particular, $\s(v)=\s(w)$, which is a contradiction.
\end{proof}

\subsection{Expansion algorithm}
\label{sub23}
In this subsection we introduce the \textit{expansion algorithm}, a procedure that can be used to create a path consisting of single-site updates from any suitable initial Potts configuration on $\L$ to one of the stable configurations. This expansion algorithm is presented in Proposition~\ref{prop:expansion} below and will be used twice: first to construct a reference path $\o^*$ between any pair of stable configurations with a prescribed height $\Phi_{\o^*}$ (Proposition~\ref{prop:refpath}) and later to show that every Potts configuration on $\L$ can be \textit{reduced} to a stable configurations with a maximum energy gap strictly smaller than $\G(\L)$, proving Theorem~\ref{thm:chadw}(ii).

A Potts configuration $\s \in \cX$ on $\L$ is a suitable initial configuration for the expansion algorithm if there exists a monochromatic bridge in $\s$, which in view of Lemma~\ref{lem:zerowastage} is equivalent to require that there exists either a column $c$ (or a row $r$) of $\L$ with $\DH_c(\s)=0$ ($\DH_r(\s)=0$, respectively). Figure~\ref{fig:suitableconfigurations} shows a few examples of suitable starting configurations. The procedure will then gradually ``expand'' this monochromatic bridge by changing the color of the vertices in the adjacent columns until the corresponding stable configuration is obtained: this is the reason why we choose to name it \textit{expansion algorithm}.

We remark that the fact that our algorithm makes a cluster grow gradually column by column (or row by row) is not crucial, and in fact we could have defined a more general expansion algorithm that leverages the vertex-isoperimetric order for grid graphs, which is known both for periodic and open boundary conditions~\cite{AC95,Bollobas1990,R98,Wang1977a}. We choose to present here a procedure based on the row and column structure of $\L$ as it is more intuitive and eventually yields the same energy bounds. 

\begin{figure}[!h]
	\centering
	\subfloat[Configurations with a monochromatic (black) bridge on column $c_0$]{\includegraphics{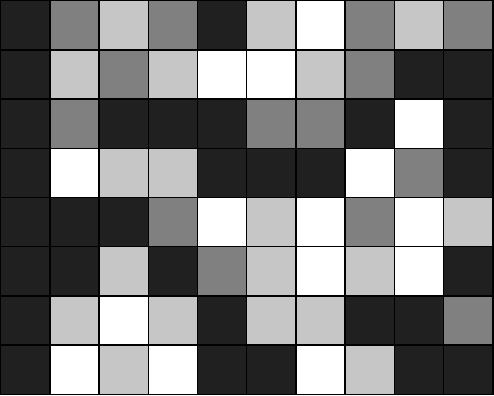}}
	\hspace{0.6cm}
	\subfloat[Configuration with a monochromatic (gray) bridge on column $c_7$]{\includegraphics{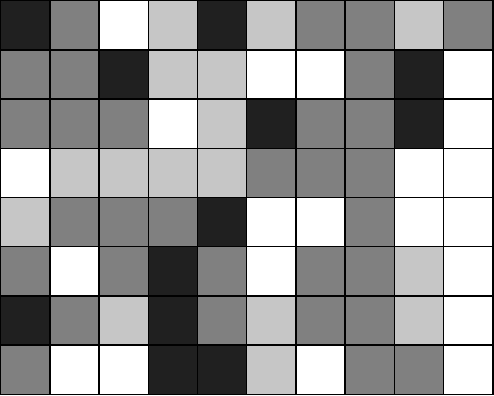}}
	\hspace{0.6cm}
	\subfloat[Configuration with a monochromatic (white) bridge on column $c_4$]{\includegraphics{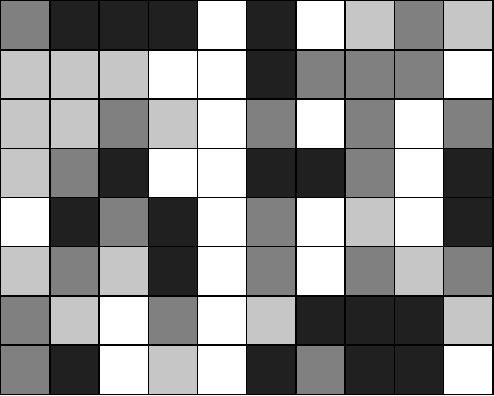}}
	\caption{Examples of suitable starting configurations for the expansion algorithm on a $8 \times 10$ grid}
	\label{fig:suitableconfigurations}
\end{figure}
\FloatBarrier

The following proposition summarizes our findings for both types of boundary conditions.
\begin{prop}[Expansion algorithm for grid graphs]
\label{prop:expansion}
Let $\s \in \cX$ be a Potts configuration on a grid graph $\L$. If $\s$ has a monochromatic $k$-bridge, then there exists a path $\o: \s \to \cc_k$ such that 
\[	
	\Phi_\o -H(\s) \leq 
	\begin{cases}
		2 & \text{ if } \L\text{ has periodic boundary conditions,}\\
		1 & \text{ if } \L\text{ has open boundary conditions.}
	\end{cases}
\]
\end{prop}

\begin{proof}
We will first describe the procedure in the case where the grid $\L$ has periodic boundary conditions, and then later show how it can be adapted for open boundary conditions.

Consider a suitable starting configuration $\s \in \cX$ for the expansion algorithm. Modulo a relabeling of the columns, we can assume that $\s$ has the monochromatic $k$-bridge on the first column $c_0$. The procedure in the case where the starting configuration $\s$ has an horizontal bridge can be easily obtained by interchanging the role of rows and columns. In what follows we associate the color black to the spin value $k$. 

We now describe an iterative procedure that builds a path $\o$ in $\cX$ from $\s$ to $\cc_k$. The path $\o$ is the concatenation of $L$ paths $\o^{(1)},\dots,\o^{(L)}$. For every $i$ along path $\o^{(i)}$ the vertices on $i$-th column are progressively colored in black. Define the \textit{intermediate configurations} $\s_{i}$, $i=0,\dots,L$, which will be the starting and ending points of the paths $\o^{(1)},\dots,\o^{(L)}$, as
\begin{equation}
\label{eq:intermediateconf}
	\s_i(v):=
	\begin{cases}
		k & \text{ if } \displaystyle v \in \bigcup_{j=0}^i c_j,\\
		\s(v) & \text{ if } \displaystyle v \in V \setminus \bigcup_{j=0}^i c_j.
	\end{cases}
\end{equation}
Clearly $\s_0 = \s$ and $\s_{L-1} = \cc_k$. For every $i=1,\dots,L$ we will now define a path $\o^{(i)}: \s_{i-1} \to \s_{i}$ of length $K$ in the following way. To help the reader following the construction we illustrate in Figure~\ref{fig:patho1} some configurations along the path $\o^{(1)}$ on a $8 \times 10$ grid.

\begin{figure}[!h]
	\centering
	\subfloat[$\o^{(1)}_{0}=\s_0$]{\includegraphics{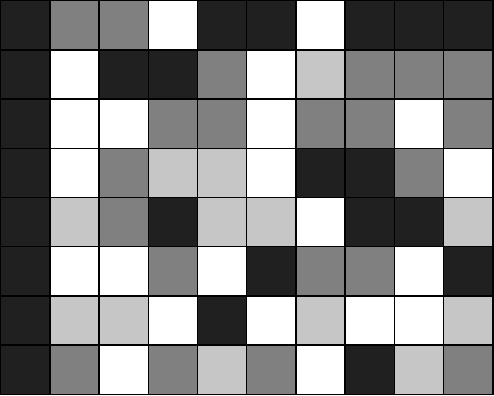}}
	\hspace{0.55cm}
	\subfloat[$\o^{(1)}_{1}$]{\includegraphics{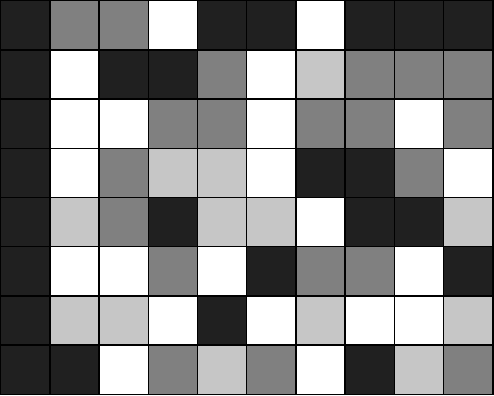}}
	\hspace{0.55cm}
	\subfloat[$\o^{(1)}_{4}$]{\includegraphics{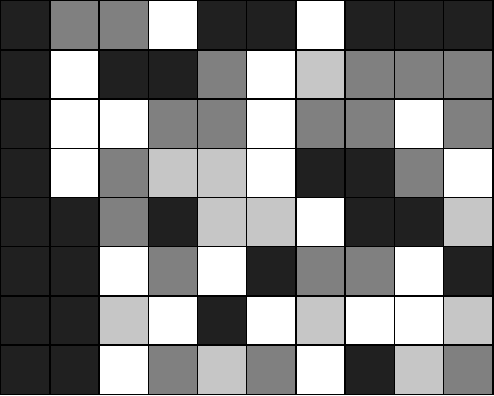}}
	\\
	\subfloat[$\o^{(1)}_{5}$]{\includegraphics{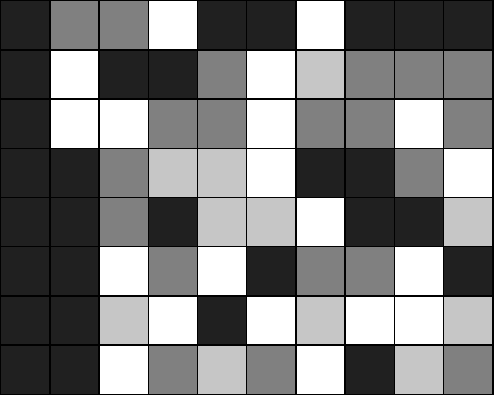}}
	\hspace{0.55cm}
	\subfloat[$\o^{(1)}_{7}$]{\includegraphics{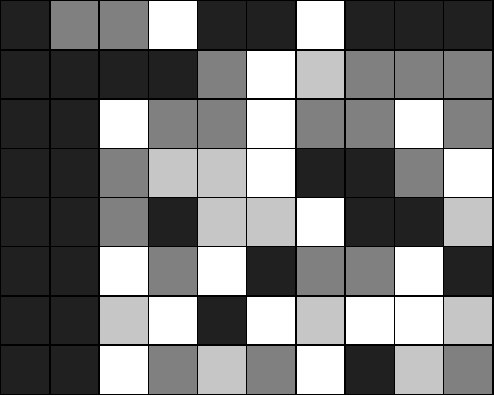}}
	\hspace{0.55cm}
	\subfloat[$\o^{(1)}_{8}=\s_1$]{\includegraphics{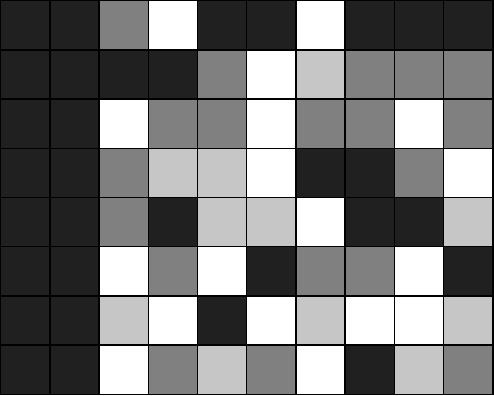}}
	\caption{Illustration of some configurations along the path $\o^{(1)}: \s_0 \to \s_1$ on a $8 \times 10$ grid}
	\label{fig:patho1}
\end{figure}
\FloatBarrier

Set $\o^{(i)}_0=\s_{i-1}$ and for any $m=1,\dots,K$ define the configuration $\o^{(i)}_{m}$ from the previous one by coloring the vertex $(i,m-1)$ as black; in other words, using the notation introduced in~\eqref{eq:svk},
\[
	\o^{(i)}_{m} := (\o^{(i)}_{m-1})^{(i,m-1),k}, \quad m=1,\dots, K.
\]
We claim that the ``energy cost'' of this single-vertex update satisfies the following inequalities:
\begin{equation}
\label{eq:energycost}
	H(\o^{(i)}_{m})-H(\o^{(i)}_{m-1}) \leq 
	\begin{cases}
		2 & \text{ if } m=1,\\
		0 & \text{ if } 1 < m < K,\\
		-2 & \text{ if } m=K.\\
	\end{cases}
\end{equation}
Note that by updating a Potts configuration on $\L$ in a single vertex $v=(i,m-1) \in \L$, the edges that can change their status (from agreeing to disagreeing and vice-versa) are only those incident to $v$. Given $\h \in \cX$ and $v \in \L$, denote by $d_v(\h)$ the number of disagreeing edges incident to vertex $v$ in configuration $\h$, \ie
\[
	d_v(\h) := \sum_{w \in \L \, : \, (v,w) \in E} \mathds{1}_{\{\h(v) \neq \h(w)\}},
\]
and rewrite the energy gap between two consecutive configurations along the path $\o^{(i)}$ as
\begin{equation}
\label{eq:disagreeing}
	H(\o^{(i)}_{m})-H(\o^{(i)}_{m-1}) = d_v(\o^{(i)}_{m}) - d_v(\o^{(i)}_{m-1}).
\end{equation}
If the considered vertex $v=(i,m-1)$ is already black in the starting configuration $\s$, the step is void and trivially $H(\o^{(i)}_{m})-H(\o^{(i)}_{m-1})=0$. Assume then that the vertex $v=(i,m-1)$ is not black, \ie $\s(v)\neq k$. Using identity~\eqref{eq:disagreeing}, the claim in~\eqref{eq:energycost} can be proved case by case. The three different cases are illustrated below in Figure~\ref{fig:patho1}, respectively in (a)-(b) for $m=1$, (c)-(d) for $1<m<K$, and (e)-(f) for $m=K$.

\begin{itemize}
	\item If $m=1$, then
	$
		d_v(\o^{(i)}_{m-1}) \geq 1,
	$
	since $v$ is not black and as such it disagrees at least with its left neighbor $(i-1,0)$ (that is black by construction), and
	$
		d_v(\o^{(i)}_{m}) \leq 3,
	$
	since $d_v(\o^{(i)}_{m})\neq 4$ in view of the fact that at least the left neighbor $(i-1,0)$ of $v$ is of the same color (\ie black).
	\item If $1 < m < K$, then
	$
		d_v(\o^{(i)}_{m-1}) \geq 2,
	$
	since $v$ is not black and as such it has a color disagreeing at least with its left neighbor $(i-1,m-1)$ and bottom neighbor $(i,m-2)$ (that are black by construction), and
	$
		d_v(\o^{(i)}_{m}) \leq 2,
	$
	since $v$ is black and agrees at least with its left neighbor $(i-1,m-1)$ and bottom neighbor $(i,m-2)$, that are both black by construction.
	\item If $m=K$, then
	$
		d_v(\o^{(i)}_{m-1}) \geq 3,
	$
	since $v$ is not black and as such it disagrees at least with its left, top, and bottom neighbors (that are black by construction), and
	$
		d_v(\o^{(i)}_{m}) \leq 1,
	$
	since $v$ is black and agrees at least with its left, top, and bottom neighbors, that all three black by construction.
\end{itemize}

For every $i=1,\dots,L-1$, the inequalities~\eqref{eq:energycost} for the energy differences along each path $\o^{(i)}$ imply that $	\Phi_{\o^{(i)}} - H(\s_{i-1}) \leq 2$. Therefore, by concatenating all the paths $\o^{(1)}, \dots, \o^{(L-1)}$ we obtain a path $\o: \s \to \cc_k$ such that
$
	\Phi_\o -H(\s) \leq 2.
$\\

We now describe how the expansion algorithm works when $\L$ has open boundary conditions. Instead of giving a full description of the procedure, we will only briefly explain the main differences from the one we just described for periodic boundary conditions.

Consider a suitable starting configuration $\s \in \cX$ to start the expansion algorithm, \ie a configuration displaying a black bridge. 
As before, it is enough to describe the procedure in the case of a vertical bridge. There are two tweaks necessaries to adapt the algorithm described earlier to this scenario:
\begin{enumerate}
	\item The columns of a grid graph with open boundary conditions are not identical and thus, differently from what we did earlier, we cannot assume without loss of generality that the monochromatic bridge lies on the first column $c_0$. Let $c^*$ be the column where the monochromatic bridge lies in configuration $\s$. The procedure described previously can be used to expand the monochromatic bridge first to the right of $c^*$, until the open boundary of $\L$ is reached, and then to the left of $c^*$ (``mirroring'' the moves described earlier) until the left open boundary of $\L$.
	\item Every new column is started by updating its bottom-most vertex, which in an grid graph with open boundaries has at most $3$ neighbors, and is completed by updating the topmost vertex, which also has at most $3$ neighbors in this case. By revisiting the previous energy costs calculations, we can derive that along any path $\o^{(i)}$ that adds a black column next to an existing one 
\[
	H(\o^{(i)}_{m})-H(\o^{(i)}_{m-1}) \leq 
	\begin{cases}
		1 & \text{ if } m=1,\\
		0 & \text{ if } 1 < m < K,\\
		-1 & \text{ if } m=K.
	\end{cases}
\]
Therefore,
$
	H(\s_{i}) \leq H(\s_{i-1}) $ and $\Phi_{\o^{(i)}} - H(\s_{i-1}) \leq 1
$
and the path $\o$ obtained by concatenating $\o^{(1)},\dots,\o^{(L)}$ then satisfies the inequality $\Phi_\o -H(\s) \leq 1.$ \qedhere
\end{enumerate}
\end{proof}

\subsection{Reference path between stable configurations}
\label{sub24}
We will now use the expansion algorithm to build a reference path between any pair of stable configurations with a prescribed height.
\begin{prop}[Reference path]
\label{prop:refpath}
Consider the Potts model on a $K \times L$ grid $\L$. For every pair of stable configurations $\cc,\dd \in \ss$, $\cc \neq \dd$, there exists a reference path $\o^*: \cc \to \dd$ such that
\[
	 \Phi_{\o^*} -H(\cc) =  
	 \begin{cases}
	 	2 \min\{K,L\}+2 & \text{ if } \L \text{ is a grid with periodic boundary conditions,}\\
		 \min\{K,L\}+1  & \text{ if } \L \text{ is a grid with open boundary conditions.}
	 \end{cases}
\]
\end{prop}
\begin{proof}
We first prove the result in the case of periodic boundary conditions and assuming $K \leq L$. When $K > L$, the construction of the reference path is similar and can be obtained by interchanging the role of rows and columns. The proof of the result when $\L$ has open boundary conditions is discussed later.

Let $\s^*$ be the configuration that agrees with the target configuration $\dd$ on the first column $c_0$ and elsewhere with the starting configuration $\cc$, \ie
\begin{equation}
\label{eq:sstar}
	\s^*(v):=
	\begin{cases}
		\dd(v) & \text{ if } v \in c_0,\\
		\cc(v) & \text{ otherwise.}
	\end{cases}
\end{equation}
We will construct a reference path $\o^*$ from $\cc$ to $\dd$ such that
\[
	\Phi_{\o^*} -H(\cc) =  2K+2
\]
as the concatenation of two paths, $\o^{(1)}: \cc \to \s^*$ and $\o^{(2)}: \s^* \to \dd$ such that
\[
	\Phi_{\o^{(1)}} = H(\cc) + 2 K \quad \text{ and } \quad \Phi_{\o^{(2)}} = H(\cc) + 2 K + 2.
\]
For simplicity we color to the vertices whose spins agree with $\cc$ as white and the one agreeing with $\dd$ as black. Figure~\ref{fig:cctodd} should help the reader following the construction of the reference path.

The path $\o^{(1)}$ is the path $(\o^{(1)}_0,\dots,\o^{(1)}_K)$ of length $K$ starting from $\o^{(1)}_0=\cc$ and obtained iteratively by coloring at step $i$ vertex $(0,i-1)$ as black. It is easy to check that
\begin{equation}
\label{eq:DHo1}
	H(\o^{(1)}_i) - H(\o^{(1)}_{i-1}) =
	\begin{cases}
		4 & \text{ if } i=1,\\
		2 & \text{ if } i=2,\dots,K-1,\\
		0 & \text{ if } i=K.
	\end{cases}
\end{equation}

\captionsetup[subfigure]{labelformat=empty}
\begin{figure}[!t]
\centering
	\subfloat[$\cc$]{\includegraphics{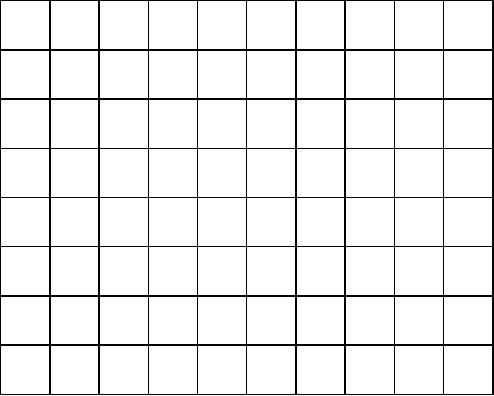}}
	\hspace{0.5cm}
	\subfloat[$\o^{(1)}_1$]{\includegraphics{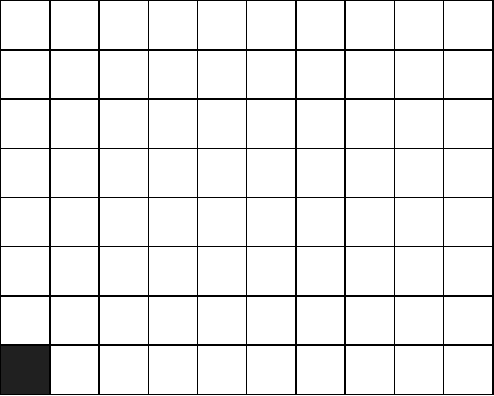}}
	\hspace{0.5cm}
	\subfloat[$\o^{(1)}_2$]{\includegraphics{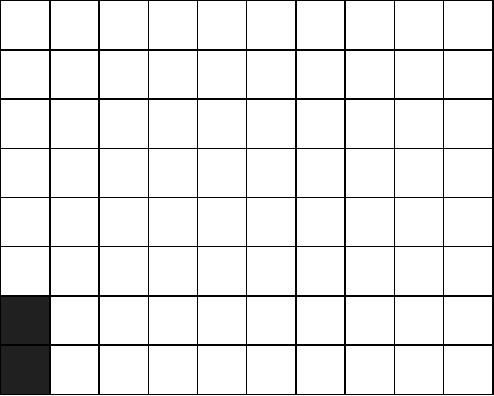}}
	\\
	\subfloat[$\o^{(1)}_{K-1}$]{\includegraphics{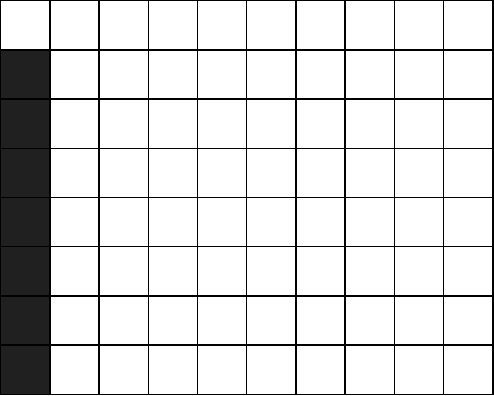}}
	\hspace{0.5cm}
	\subfloat[$\s^*=\o^{(1)}_K$]{\includegraphics{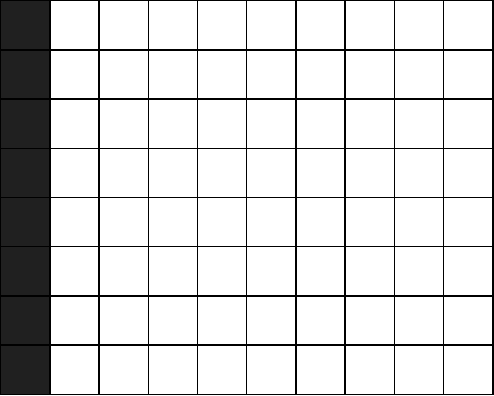}}
	\hspace{0.5cm}
	\subfloat[$\o^{(2)}_1$]{\includegraphics{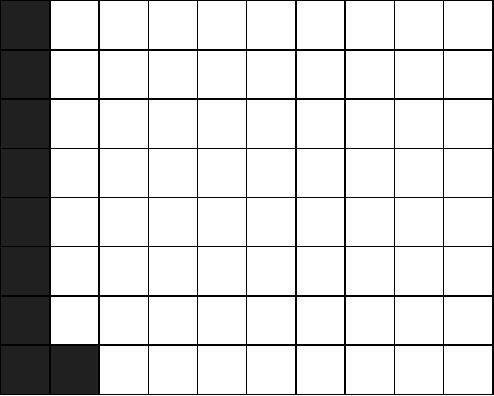}}
	\\
	\subfloat[$\o^{(2)}_2$]{\includegraphics{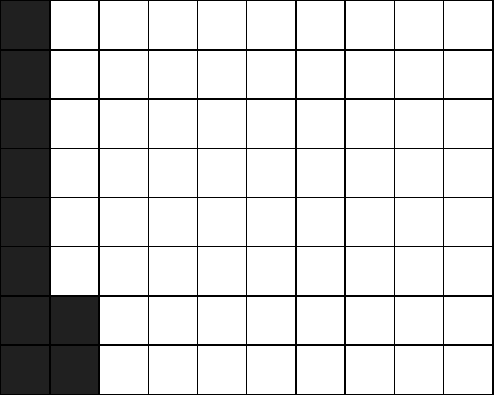}}
	\hspace{0.5cm}
	\subfloat[$\o^{(2)}_K$]{\includegraphics{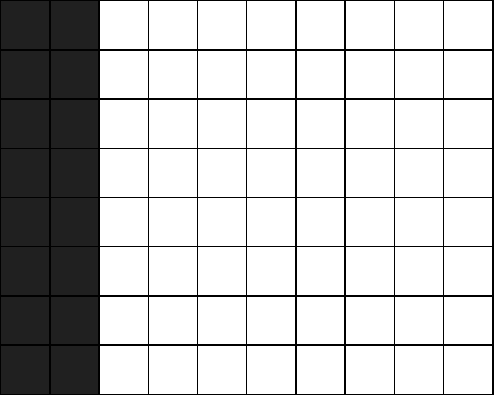}}
	\hspace{0.5cm}
	\subfloat[$\o^{(2)}_{K(L-2)}$]{\includegraphics{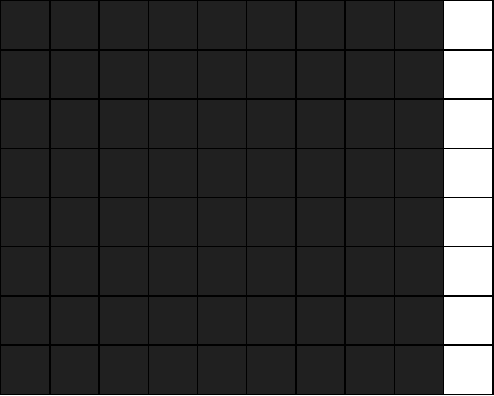}}
	\\
	\subfloat[$\o^{(2)}_{K(L-2)+1}$]{\includegraphics{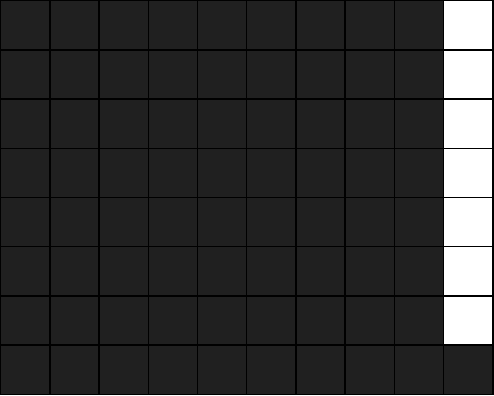}}
	\hspace{0.5cm}
	\subfloat[$\o^{(2)}_{K(L-2)+2}$]{\includegraphics{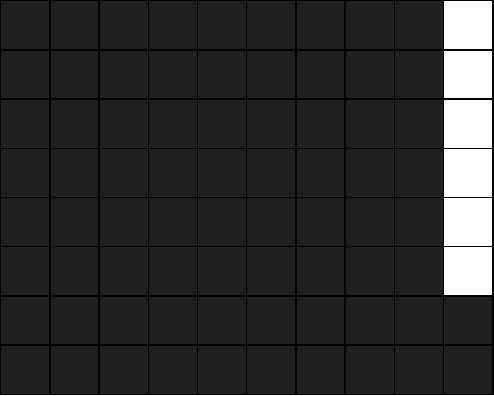}}
	\hspace{0.5cm}
	\subfloat[$\dd$]{\includegraphics{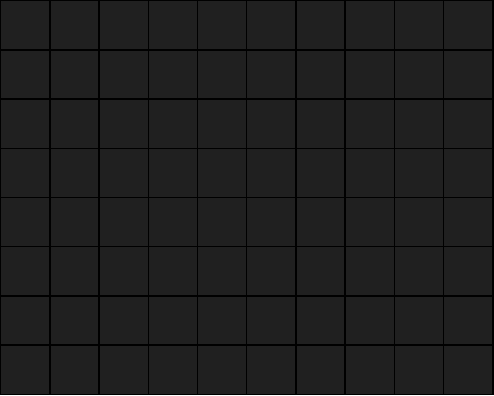}}
	\caption{Illustration of the reference path $\o^*: \cc \to \dd$ in the case $K\leq L$}
	\label{fig:cctodd}
\end{figure}
\captionsetup[subfigure]{labelformat=parens}

Indeed, coloring in black the first vertex $(0,0)$ creates new disagreements with its $4$ white neighbors, while for each $i=2,\dots,K-1$, the black coloring of vertex $(0,i-1)$ creates $3$ new disagreements, but resolves $1$ with respect to the previous configuration, so that the total amount of disagreements increases by $2$. The last vertex $(0,K-1)$ colored in black to obtain configuration $\o^{(1)}_{K}$ resolves $2$ disagreements and create $2$ new ones, resulting in a zero net energy difference with respect to $\o^{(1)}_{K-1}$. A schematic illustration of the path $\o^{(1)}$ can be found in the first five snapshots of Figure~\ref{fig:cctodd}. In view of~\eqref{eq:DHo1}, the configuration with the highest energy along $\o^{(1)}$ are $\o^{(1)}_{K-1}$ and $\o^{(1)}_{K}=\s^*$, since
\begin{equation}
\label{eq:DHsstar}
	\DH(\o^{(1)}_{K-1})=2 K=\DH(\s^*),
\end{equation}
and therefore $\Phi_{\o^{(1)}} = H(\s^*) = H(\cc) + 2 K$.

The newly obtained configuration $\s^*$ has a monochromatic black bridge on $c_0$ and as such is a suitable starting configuration for the expansion algorithm introduced earlier. In view of Proposition~\ref{prop:expansion}, such an algorithm outputs a path $\o^{(2)}: \s^* \to \dd$ such that 
\[
	\Phi_{\o^{(2)}} = H(\s^*) +2 \stackrel{\eqref{eq:DHsstar}}{=} H(\cc) + 2 K + 2.
\]

In the case where $\L$ has open boundary conditions, there is no need to define a different reference path, since the exact same reference path yields the desired identity.  The only thing one needs to do is reviewing the calculations for the maximum energy along the paths $\o^{(1)}$ and $\o^{(2)}$, which are now different in view of the open boundary conditions. More specifically the fact that the vertices in $c_0$ have no left neighbors and the properties of the expansion algorithm for grids with open boundary conditions (see Proposition~\ref{prop:expansion}) yield 
\[
	\Phi_{\o^{(1)}} = H(\cc) + K \quad \text{ and } \quad \Phi_{\o^{(2)}} = H(\cc) + K + 1.
\]
from which the conclusion readily follows.
\end{proof}
\FloatBarrier

\subsection{Communication energy between stable configurations}
\label{sub25}

Given a configuration $\s \in \cX$ and a spin value $k \in \{1,\dots,q\}$, let $B_k(\s) \in \N \cup \{0\}$ be the total number of $k$-bridges (horizontal and vertical) that configuration $\s$ has. The next lemma shows how this quantity evolves with single-spin updates and relates its increments with geometric properties of the spin configurations.

\begin{lem}[Bridges creation and deletion]
\label{lem:Bk}
Let $\s,\s' \in \cX$ be two Potts configuration that differ by a single-spin update, that is~$\left |\{v \in V : \s(v)\neq \s'(v)\} \right |=1$. Then for every $k \in \{1,\dots,q\}$ we have that
\[
	B_k(\s')-B_k(\s) \in \{-2,-1,0,1,2\},
\]
and $B_k(\s')-B_k(\s) = 2$ if and only if $\s'$ a $k$-cross that $\s$ does not have.
\end{lem}
This lemma basically states that at most two bridges of a given color can be created or destroyed by a single-spin update and that if exactly two bridges are created together, then they must be orthogonal (one horizontal and one vertical).
\begin{proof}
The proof revolves around the simple observation that a single-spin update can create (or destroy) a bridge only in the row and/or in the column where it lies. Hence, by updating the spin of a given vertex to $k$, at most two $k$-bridges can be simultaneously created or destroyed. One implication of the second statement is trivial; for the converse one, observe that if exactly two $k$-bridges are created by a single-spin update, then they cannot be both horizontal or both vertical, and thus they intersect creating a $k$-cross.
\end{proof}

\begin{prop}[Communication energy lower bound]
\label{prop:lowerbound}
Consider the Potts model on a $K \times L$ grid with $\max\{K,L\} \geq 3$. Then, for every $\cc,\dd \in \ss$, with $\cc \neq \dd$, the following inequality holds
\begin{equation}
\label{eq:ineqch}
	 \Phi(\cc,\dd) -H(\cc) \geq  
	 \begin{cases}
	 2 \min\{K,L\}+2 & \text{ if } \L \text{ has periodic boundary conditions,}\\
	 \min\{K,L\}+1  & \text{ if } \L \text{ has open boundary conditions.}
	 \end{cases}
\end{equation}
\end{prop}
\begin{proof}
Consider first the case where $\L$ has periodic boundary conditions. It is enough to show that along every path $\o: \cc \to \dd$ in $\cX$ there exists at least one configuration with energy gap not smaller than $2 \min\{K,L\}+2$. Let $k \in \{1,\dots,q\}$ be the spin value such that $\dd = \cc_k$. In the rest of the proof we will associate the color black to the spin value $k$ and in particular we will refer to $k$-bridges and $k$-crosses as black bridges and black crosses, respectively.

Consider a path $\o$ from $\cc $ to $\dd$ of length $n$, so that $\o=(\o_1,\dots,\o_n)$ with $\o_1=\cc$ and $\o_n=\dd$. Note that $\cc$ has no black bridges, \ie $B_k(\cc)=0$, while $\dd$ is has $B_k(\dd)=K+L$ black bridges. Hence, there exists a configuration along the path $\o$ that is the first to have at least two black bridges; let $m^* \in \N$ be the corresponding index, \ie
$
	m^*:=\min \{ m \leq n ~|~ B_k(\o_m) \geq 2 \}.
$
Consider the configuration $\o_{m^*-1}$ that precedes $\o_{m^*}$ in the path $\o$. We claim that the total energy gap of the configuration $\o_{m^*-1}$ satisfies the following inequality
\begin{equation}
\label{eq:Um}
	\DH(\o_{m^*-1}) \geq 2 \min\{K,L\}+2.
\end{equation}
\newpage
We prove this claim by considering separately three scenarios:
\begin{itemize}
	\item[\textup{(a)}] $\o_{m^*}$ displays only vertical black bridges;
	\item[\textup{(b)}]  $\o_{m^*}$ displays only horizontal black bridges;
	\item[\textup{(c)}]  $\o_{m^*}$ displays at least one black cross.
\end{itemize}

Consider first scenario (a), thus assuming that $\o_{m^*}$ displays only vertical black bridges. From the definition of $m^*$, it follows that $B_k(\o_{m^*-1})\leq 1$ and $B_k(\o_{m^*}) \geq 2$. Furthermore, the difference $B_k(\o_{m^*})-B_k(\o_{m^*-1})$ must be strictly smaller than $2$, since otherwise $\o_{m^*}$ would have a black cross in view of Lemma~\ref{lem:Bk}. Hence, 
\[
	B_k(\o_{m^*-1})=1 \quad \text{ and } \quad B_k(\o_{m^*})=2.
\]
Hence, configuration $\o_{m^*}$ has exactly two vertical black bridges, say on columns $c$ and $c'$, see an example in Figure~\ref{fig:phi_a}. \vspace{-0.5cm}
\begin{figure}[!h]
	\centering
	\subfloat[Configuration $\o_{m^*-1}$ with one black bridge on column $c$ and an incomplete one on column $c'$]{\includegraphics{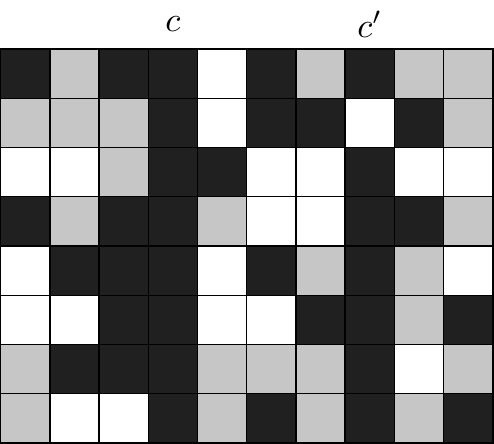}}
	\hspace{2cm}
	\subfloat[Configuration $\o_{m^*}$ with two black bridges on column $c$ and $c'$]{\includegraphics{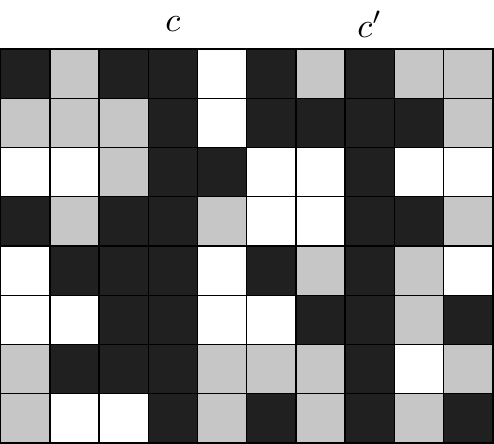}}
	\caption{Illustration of scenario (a)}
	\label{fig:phi_a}
\end{figure}
\FloatBarrier
Since $\o_{m^*-1}$ and $\o_{m^*}$ differ by a single-spin update and $B_k(\o_{m^*-1})=1$, it follows that configuration $\o_{m^*-1}$ has only one vertical $k$-bridge, say on column $c$, while it has all black vertices but one on column $c'$. In particular, $\o_{m^*-1}$ has no vertical bridge on column $c'$ and therefore, in view of Lemma~\ref{lem:zerowastage}(d),
\begin{equation}
\label{eq:Uva}
	\DH_{c'}(\o_{m^*-1}) \geq 2.
\end{equation}
We claim that $\o_{m^*-1}$ cannot have any horizontal bridge. Indeed:
\begin{itemize}
	\item the presence of a black horizontal bridge in some row would imply that $B_k(\o_{m^*-1}) \geq 2$ (since $\o_{m^*-1}$ has by construction at least a vertical black bridge on column $c$), contradicting the definition of $m^*$;
	\item there cannot be non-black horizontal bridges either in view of the black bridge in column $c$ and Lemma~\ref{lem:monochromaticcross}.
\end{itemize} 
The absence of horizontal bridges together with Lemma~\ref{lem:zerowastage}(c) then yields that $\DH_r(\o_{m^*-1}) \geq 2$ for every row $r$ and thus
\begin{equation}
\label{eq:Uha}
	\sum_{i=0}^{K-1} \DH_{r_i} (\o_{m^*-1}) \geq 2 K.
\end{equation}
In view of identity~\eqref{eq:DHhv}, inequalities~\eqref{eq:Uva} and~\eqref{eq:Uha} together yields
\begin{equation}
\label{eq:Ua}
	\DH(\o_{m^*-1}) \geq \DH_{c'}(\o_{m^*-1}) + \sum_{i=0}^{K-1} \DH_{r_i} (\o_{m^*-1}) \geq 2K+2.
\end{equation}

In scenario (b) we can argue like in (a) but interchanging the role of rows and columns, and obtain the following inequality
\[
	\DH(\o_{m^*-1}) \geq 2 L +2.
\]

Consider now scenario (c), where we assume $\o_{m^*}$ displays at least one black cross. By definition of $m^*$, the quantity $B_k(\o_{m^*-1})$ can take only two values, $0$ or $1$, and we consider these two cases separately.\\

Assume first that $B_k(\o_{m^*-1})=0$, which means that $\o_{m^*-1}$ has no black bridges, see an example in Figure~\ref{fig:phi_c1}. \vspace{-0.5cm}
\begin{figure}[!h]
	\centering
	\subfloat[Configuration $\o_{m^*-1}$]{\includegraphics{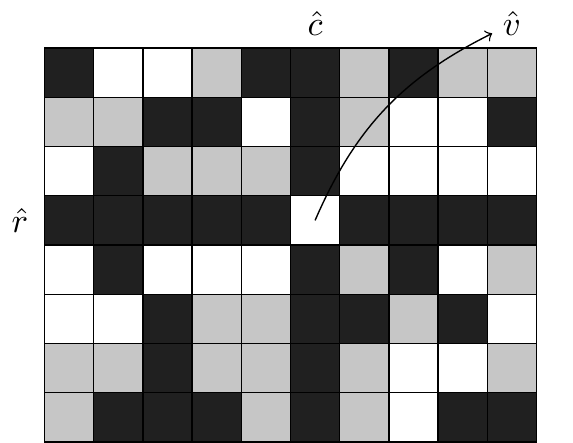}}
	\hspace{2cm}
	\subfloat[Configuration $\o_{m^*}$]{\includegraphics{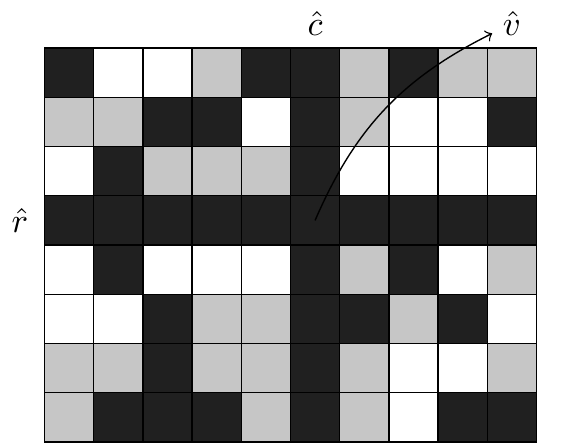}}
	\caption{Illustration of scenario (c) where configuration $\o_{m^*-1}$ is such that $B_k(\o_{m^*-1})=0$}
	\label{fig:phi_c1}
\end{figure}
\FloatBarrier
Since $\o_{m^*-1}$ and $\o_{m^*}$ differ by a single-spin update, Lemma~\ref{lem:Bk} gives that $B_k(\o_{m^*})\geq 2$ and thus we can conclude that $B_k(\o_{m^*})=2$, in view of the definition of $m^*$.

Lemma~\ref{lem:Bk} implies further that $\o_{m^*}$ displays a unique black cross. Let $\hat{r}$ and $\hat{c}$ be respectively the row and the column on which such cross lies. Since $B_k(\o_{m^*-1})=0$, the horizontal and vertical black bridges that $\o_{m^*}$ has must have been created simultaneously from configuration $\o_{m^*-1}$ by updating the spin in the vertex, say $\hat{v}$, where $\hat{r}$ and $\hat{c}$ intersect. Hence, by construction,
\[
	\o_{m^*-1}(v) = k \quad \forall \, v \in  \hat{r} \cup \hat{c}, \, v \neq \hat{v}.
\]
Since there is a black vertex in every row and in every column, configuration $\o_{m^*-1}$ cannot have non-black (horizontal or vertical) bridges. This fact, together with our assumption that $B_k(\o_{m^*-1})=0$, implies that $\o_{m^*-1}$ has \textit{no} bridges of \textit{any} color, i.e.
\[
	B_l(\o_{m^*-1})=0 \quad \forall \, l \in \{1,\dots,q\}.
\]
Therefore, thanks to Lemma~\ref{lem:zerowastage}(c) and (d), the energy gap is not smaller than $2$ in every row and column, and, hence,
\[
		\sum_{i=0}^{K-1} \DH_{r_i} (\o_{m^*-1}) \geq 2 K \quad \text{ and } \quad \sum_{j=0}^{L-1} \DH_{c_j} (\o_{m^*-1}) \geq 2 L.
\]
In view of identity~\eqref{eq:DHhv}, the latter two inequalities yield
\[
	\DH(\o_{m^*-1}) \geq  2K +2L \geq 2\min\{K,L\} + 2 \max\{K,L\} > 2\min\{K,L\} + 2.
\]

Consider now the scenario in which $B_k(\o_{m^*-1})=1$, which means that $\o_{m^*-1}$ has a unique black bridge, see Figure~\ref{fig:phi_c2} for an example.
\vspace{-0.5cm}
\begin{figure}[!h]
	\centering
	\subfloat[Configuration $\o_{m^*-1}$]{\includegraphics{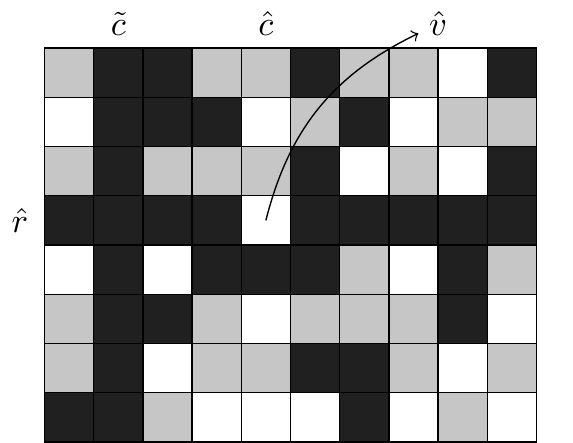}}
	\hspace{2cm}
	\subfloat[Configuration $\o_{m^*}$]{\includegraphics{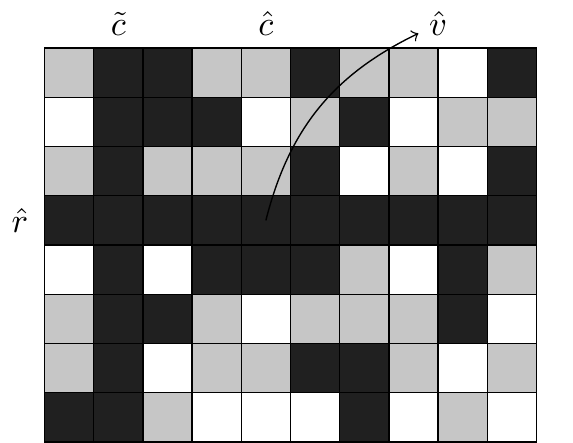}}
	\caption{Illustration of scenario (c) where configuration $\o_{m^*-1}$ is such that $B_k(\o_{m^*-1})=1$}
	\label{fig:phi_c2}
\end{figure}
\FloatBarrier
We will assume that such a black bridge is vertical and that lies in column $\tilde{c}$; if instead it is horizontal, the proof is identical after interchanging the role of rows and columns and leads precisely to the same lower bound for $\DH(\o_{m^*-1})$. By virtue of Lemma~\ref{lem:monochromaticcross}, the presence of the vertical black bridge in column $\tilde{c}$ makes impossible the existence of any horizontal non-black bridge in configuration $\o_{m^*-1}$. Furthermore, by assumption $\o_{m^*-1}$ has no horizontal black bridges and Lemma~\ref{lem:zerowastage}(c) then yields
\begin{equation}
\label{eq:Uhc}
	\sum_{i=0}^{K-1} \DH_{r_i} (\o_{m^*-1}) \geq 2 K.
\end{equation}
Since $\o_{m^*-1}$ and  $\o_{m^*}$ differ by a single-spin update, the presence of a black cross $\o_{m^*}$ and the absence of horizontal black bridges in $\o_{m^*-1}$ imply that $\o_{m^*}$ has a unique horizontal black bridge, say on row $\hat{r}$. By construction, the vertex, say $\hat{v}$, where $\o_{m^*}$ and $\o_{m^*-1}$ differ must lie in such a row, and
\[
	\o_{m^*-1}(\hat{v}) \neq k \quad \text{ and } \quad \o_{m^*-1}(v) = k \quad \forall \, v \in \hat{r}, v\neq \hat{v}
\]
Let $\hat{c}$ be the column where $\hat{v}$ lies. The black vertices in row $\hat{r}$ implies that configuration $\o_{m^*-1}$ has no vertical $l$-bridge with $l\neq k$ in every column $c \neq \hat{c},\tilde{c}$. Lemma~\ref{lem:zerowastage}(d) then yields that in each of these $L-2$ columns the energy gap is greater than or equal to $2$ and thus
\begin{equation}
\label{eq:Uvc}
	\sum_{j=0}^{L-1} \DH_{c_j} (\o_{m^*-1}) \geq 2(L-2) = 2 L - 4.
\end{equation}
From inequalities~\eqref{eq:Uhc} and~\eqref{eq:Uvc} it follows that
\[
	\DH(\o_{m^*-1}) \geq 2K+2L-4 \geq 2\min\{K,L\} + 2 \max\{K,L\} - 4 \geq 2\min\{K,L\} + 2,
\]
where the last inequality holds since $\max\{K,L\} \geq 3$.\\

The proof of inequality~\eqref{eq:ineqch} in the case where $\L$ has open boundary condition is very similar and thus omitted. The only tweak necessary is easy to explain and is a consequence of the fact that the lower bound for the energy gap on rows or columns without bridges is different due to the open boundary conditions. As illustrated in Lemma~\ref{lem:zerowastage}, a row or column without bridges has energy gap not smaller than $2$ when $\L$ has periodic boundary conditions, while we only know that is non-zero (and in particular greater than or equal to $1$) when $\L$ has open boundary conditions. By adjusting this factor in all the inequalities derived above, one gets the desired lower bound for the communication energy $\Phi(\cc,\dd)$ for the case of open boundary conditions, which is precisely half of that obtained in the case of periodic boundary conditions.
\end{proof}

\subsection{Proof of Theorem~\ref{thm:chadw}}
\label{sub26}
In this subsection we combine the results obtained in the previous subsections and prove the structural properties of the energy landscape that have been presented in Theorem~\ref{thm:chadw}.\\

The proof of statement (i) readily follows by combining the reference path constructed in Propositions~\ref{prop:refpath} (which yields an upper bound for $\Phi(\cc,\dd)$) with the matching lower bound obtained in Proposition~\ref{prop:lowerbound}.\\

We now focus on the proof of statement (ii). We first prove the result for grids (a) with periodic boundary conditions and then later (b) with open boundary conditions.\\

(a) Consider a configuration $\s \in \cX \setminus \ss$. If configuration $\s$ has a (vertical or horizontal) $k$-bridge for some $k=1,\dots,q$, then $\s$ is a suitable starting configuration for the expansion algorithm can be used to build a path $\o: \s \to \cc_k$ such that $\Phi_\o \leq H(\s) + 2$.

Consider now the opposite scenario, the one where $\s$ has no (vertical or horizontal) bridge. Take the column, say $c^*$, with the largest number of vertices of the same color, say black, and let $k$ be the associated spin value. Define
\[
	\s^*(v):=
	\begin{cases}
		\s(v) & \text{ if } v \in V \setminus c^*,\\
		k  & \text{ if } v \in c^*.
	\end{cases}
\]
Denote by $m$ the number of vertices in which configurations $\s$ and $\s^*$ differ, that is $m :=|\{v \in V ~:~ \s(v) \neq \s^*(v) \}|$. Note that $m$ is precisely the number of non-black vertices that configuration $\s$ has on column $c^*$, since
\[
	\{v \in V ~:~ \s(v) \neq \s^*(v) \} =\{v \in c^* ~:~ \s(v) \neq k \}.
\]
In particular, by construction $m <K = |c^*|$. We will define a path from $\s$ to $\s^*$ in which the $m$ non-black vertices are progressively colored in black. The order in which these vertices are updated is crucial to obtain the desired bound for $\Phi_\o$. More specifically, we build a path $\o^{(1)}: \s \to \s^*$ of length $m$ that starts at $\o^{(1)}_0= \s$ and is construct inductively as follows: for every step $i=1,\dots,m$, 
\begin{itemize}
	\item[(1)] Consider a vertex $v_i \in c^*$ such that (i) $\o^{(1)}_{i-1}(v) \neq k$ and (ii) has at least one black neighbors on column $c^*$;
	\item[(2)] Define the configuration $\o^{(1)}_i$ from $\o^{(1)}_{i-1}$ by coloring vertex $v_i$ as black, i.e.
	\[
		\o^{(1)}_i(v):=
		\begin{cases}
			\o^{(1)}_{i-1}(v) & \text{ if } v \neq v_i,\\
			k	 						& \text{ if } v=v_i.
		\end{cases}
	\]
\end{itemize}
The way in which the vertices $v_1,\dots,v_m$ of column $c^*$ are progressively chosen guarantees that for every $i=1,\dots, m-1$
\[
	\DH(\o^{(1)}_i) \leq \DH(\o^{(1)}_{i-1})+2,
\]
since at most two disagreements are created by coloring vertex $v_i$ as black, and that
\[
	\DH(\o^{(1)}_m) \leq \DH(\o^{(1)}_{m-1}),
\]
since vertex $v_m$ has by construction exactly two black neighbors on column $c^*$. Hence, the path $\o^{(1)}$ is such that $\Phi_{\o^{(1)}} - H(\s) \leq 2 (m-1)$. Configuration $\s^*$ has a vertical black bridge and thus the expansion algorithm yields a path $\o^{(2)}: \s^* \to \cc_k$ such that $\Phi_{\o^{(2)}} - H(\s^*) \leq 2$. The concatenation of  $\o^{(1)}$ and  $\o^{(2)}$ then is a path from $\s$ to $\cc_k$ that guarantees that $\Phi(\s,\ss)-H(\s) \leq 2(m-1) +2 \leq 2m < 2K < 2K + 2$.\\

(b) We consider now the case of open boundary conditions, in which we only briefly need to review the calculations already done in (a). If configuration $\s$ has a monochromatic bridge, then the expansion algorithm guarantees that $\Phi(\s,\ss)-H(\s) \leq 2$. If there is no bridge, then define the configuration $\s^*$ obtained from $\s$ by coloring as black all the vertices on the first column, i.e.
\[
	\s^*(v):=
	\begin{cases}
		\s(v) & \text{ if } v \in V \setminus c_0,\\
		k  & \text{ if } v \in c_0.
	\end{cases}
\]
As we did in (a), we will construct a path from $\s$ to $\ss$ using configuration $\s^*$ as intermediate configuration. As before, let $m$ denote the number of vertices in which configuration $\s$ and $\s^*$ differ. By progressively coloring them as black, always updating a vertex with at least one black neighboring vertex on $c_0$, the energy cost is no larger than $1$ for every vertex newly colored in black thanks to the open boundary conditions. In particular, coloring the last non-black vertex on column $c_0$ costs $0$ or less, since by construction it had at most one disagreeing neighbor. In this way we have build a path such that
\begin{equation}
\label{eq:blackbridgeopen}
	\Phi(\s,\s^*)-H(\s) \leq m-1.
\end{equation}
Having a black bridge, $\s^*$ is a suitable starting configuration for the expansion algorithm that yields a path to the stable configuration with all black vertices, obtaining in this way a path from $\s$ to $\cc_k \in \cX$. Combining inequality~\eqref{eq:blackbridgeopen} with that given by Proposition~\ref{prop:expansion}, one obtains that
\[
	\Phi(\s,\ss)-H(\s) \leq (m-1) +1 \leq m < K <K+1. \qed
\]

\section{Proof of Theorem~\ref{thm:main}}
\label{sec3}
In this section we present the proof of Theorem~\ref{thm:main}, which combines the model-independent results derived in~\cite{NZB16} with the structural properties of the energy landscape presented in Theorem~\ref{thm:chadw}.

\subsection{Asymptotic behavior of hitting times (Proof of Theorem~\ref{thm:main}(i)-(ii))}
\label{sub31}
Consider the target stable configuration $\dd \in \ss$. We first claim that
\begin{equation}
\label{eq:claim}
	 \forall \, \s\neq \dd \qquad \Phi(\s,\dd) - H(\s) \leq \G(\L).
\end{equation}
If $\s \in \ss \setminus \{\dd\}$, then the inequality follows immediately from the reference path given in Proposition~\ref{prop:refpath} in combination with Theorem~\ref{thm:chadw}(i). In the opposite case, if $\s \not\in \ss$ property~\eqref{eq:adw} in Theorem~\ref{thm:chadw}(ii) guarantees that there exists a stable configuration $\cc^* \in \ss$ such that 
\[
	\Phi(\s,\cc^*)-H(\s) < \G(\L),
\]
which means that there exists a path $\o^*: \s \to \cc^*$ with $\Phi_{\o^*}-H(\s) < \G(\L) $. If $\cc^*=\dd$, then the claim in~\eqref{eq:claim} is proved. Otherwise, we can create a path  $\o: \s \to \dd$ as concatenation of two paths, $\o^{(1)}=\o^*$ and $\o^{(2)}: \cc^* \to \dd$ as the reference path given in Proposition~\ref{prop:refpath}. It is immediate to check that the resulting path $\o: \s \to \dd$ satisfies $\Phi_\o - H(\s) \leq \G(\L) $, and thus~\eqref{eq:claim} holds also in this case.\\
In view of the inequality~\eqref{eq:claim} \cite[Proposition 3.18]{NZB16} holds and one concludes by applying~\cite[Corollary 3.16]{NZB16} and~\cite[Theorem 3.19]{NZB16}.\\

\subsection{Asymptotic exponentiality of $\tau^{\cc}_{\ss \setminus \{\cc\}}$ (Proof of Theorem~\ref{thm:main}(iii))}
\label{sub32}
Since the statement of Theorem~\ref{thm:chadw}(i) holds for any pair of stable configurations, it follows that
\[
	 \forall \, \cc \in \ss \qquad \Phi(\cc,\ss \setminus \{\cc\}) - H(\cc) = \G(\L).
\]
Combining this identity with inequality~\eqref{eq:adw} in Theorem~\ref{thm:chadw}(ii) immediately yields that
\begin{equation}
\label{eq:asymexpcond}
	 \forall \, \cc \in \ss \qquad \max_{\s \in \cX \setminus \ss} \Phi(\s,\ss)-H(\s) < \Phi(\cc,\ss \setminus \{\cc\}) - H(\cc).
\end{equation}
This inequality means that, in view of the target set $\ss \setminus \{ \cc\}$, the cycle where the starting configuration $\cc$ lies is the deepest cycle of the energy landscape $(\cX \setminus \ss) \cup \{\cc\}$ and thus the exit time from this cycle dominates the tunneling time from $\ss$ to the target set $\ss \setminus \{ \cc\}$. Applying first~\cite[Proposition 3.20]{NZB16} and then~\cite[Theorem 3.19]{NZB16} the proof is concluded.

\subsection{Asymptotic exponentiality of $\tcd$ (Proof of Theorem~\ref{thm:main}(iv))}
\label{sub33}
First of all notice that when $q=2$, statements (iii) and (iv) coincide and thus there is nothing to prove.

In the case $q>2$, although statements (iii) and (iv) look very similar, the proof of the asymptotic exponentiality of the scaled tunneling time $\tcd$ does not immediately follow from the structural properties of the energy landscape and this is the reason why it is presented separately. Indeed in this case the subset $\ss \setminus \{\cc,\dd\}$ is not empty, which means that there exists at least a third stable configuration $\eta \in \ss \setminus \{\cc.\dd\}$ such that
\[
	\Phi(\cc,\dd) - H(\cc) \not >  \Phi(\eta, \dd) - H(\eta),
\]
as both the left-hand and right-hand sides are equal to $\G(\L)$ by Theorem~\ref{thm:chadw}. This means that when the target state is a precise stable configuration $\dd$, the condition analogous to~\eqref{eq:asymexpcond} does not hold anymore, as the energy landscape $\cX \setminus \{\dd\}$ has several equally deep cycles and not a unique one as in Subsection~\ref{sub32}. The Markov chain may be trapped in any of these cycles  and the exit times from them do not vanish in the limit $\binf$ and therefore they must be considered to determine the asymptotic distribution of $\tcd / \E \tcd$. From a technical standpoint, the fact that the condition analogous to~\eqref{eq:asymexpcond} does not hold in this case means that we cannot apply directly~\cite[Proposition 3.20]{NZB16} and~\cite[Theorem 3.19]{NZB16} as we did in the previous subsection.

The proof of the asymptotic exponentiality of $\tcd$ is thus obtained leveraging Theorem~\ref{thm:main}(iii) in combination with a stochastic representation of the tunneling time $\tcd$ that exploits the intrinsic symmetries of the energy landscape $(\cX,H,Q)$ corresponding to the $q$-state Potts model on $\L$. 

For any $k,l \in \{1,\dots, q\}$, $k\neq l$, define $\Psi_{k,l}: \cX \to \cX$ the mapping that associate to a configuration $\s \in \cX$ another configuration $\s' =\Psi_{k,l}(\s)$ such that
\begin{equation}
\label{eq:automorphism}
	\s'(v) = 
	\begin{cases}
		\s(v) & \text{ if } \s(v) \neq k,l, \\
		k & \text{ if } \s(v) =l, \\
		l & \text{ if } \s(v) =k.
	\end{cases}	
\end{equation}
The configuration $\Psi_{k,l}(\s)$ is thus obtained from $\s$ by interchanging every spin with value $k$ with a spin with value $l$ and vice-versa, while leaving all the other $q-2$ spin values unchanged. In other word, the automorphism $\Psi_{k,l}$ swaps two colors (corresponding to spins $K$ and $l$) while keeping the remaining $q-2$ colors fixed, see an example in Figure~\ref{fig:automorphism} for the $4$-state Potts model using the color convention $\{1,2,3,4\} \longleftrightarrow \{\tikz\draw[fill=white] (0,0) circle (.75ex); , \, \tikz\draw[fill=gray!45!white] (0,0) circle (.75ex); , \, \tikz\draw[fill=gray] (0,0) circle (.75ex); , \, \tikz\draw[fill=black!75!gray] (0,0) circle (.75ex); \}$.

\begin{figure}[!ht]
	\centering
	\subfloat[A Potts configuration $\s$]{\makebox[1.2\width]{\includegraphics{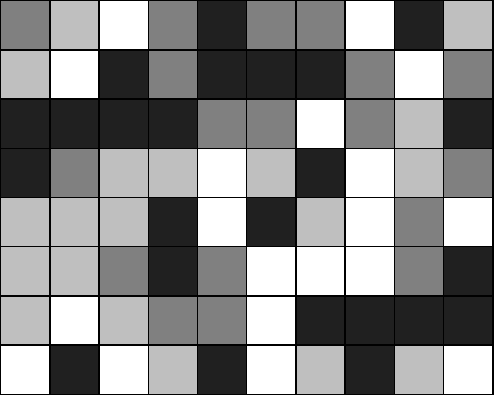}}}
	\hspace{1.2cm}
	\subfloat[The Potts configuration $\Psi_{1,4}(\s)$, where white and black colors have been inverted]{\makebox[1.2\width]{\includegraphics{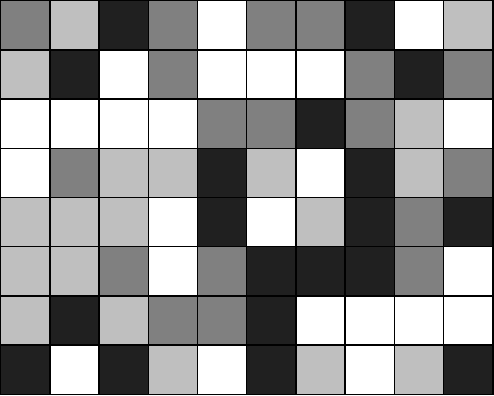}}}
	\\
	\subfloat[The Potts configuration $\Psi_{2,4}(\s)$, where light gray and black colors have been inverted]{\makebox[1.2\width]{\includegraphics{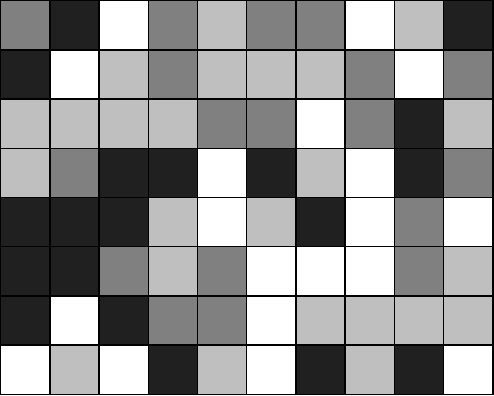}}}
		\hspace{1.2cm}
	\subfloat[The Potts configuration $\Psi_{3,4}(\s)$, where gray and black colors have been inverted]{\makebox[1.2\width]{\includegraphics{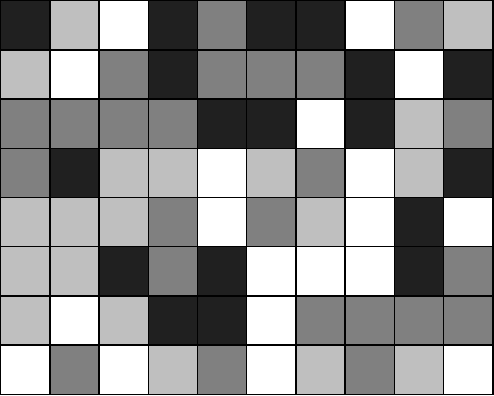}}}
	\caption{Examples of the automorphisms  $\Psi_{1,4}$,  $\Psi_{2,4}$, and  $\Psi_{3,4}$ for the Potts model with $q=4$ on the $8 \times 10$ grid}
	\label{fig:automorphism}
\end{figure}

Exploiting the family of automorphisms $ \{ \Psi_{k,l}\}_{k,l=1,\dots , q}$ and arguing like in~\cite[Proposition 2]{Zocca2017b} or in~\cite[Proposition 4.1]{Zocca2017} we can construct a coupling between different copies of the Markov chain $\xtbb$ and show that for any $\cc \in \ss$ and at any temperature $\b>0$ the following properties hold:
\begin{itemize}
	\item[\textup{(i)}] The random variable $X_{\tau^{\cc}_{\ss \setminus \{\cc\}}}$ has a uniform distribution over $\ss \setminus \{\cc\}$;
	\item[\textup{(ii)}] The distribution of the random variable $\tau^{\cc}_{\ss \setminus \{\cc\}}$ does not depend on $\cc$;
	\item[\textup{(iii)}] The random variables $\tau^{\cc}_{\ss \setminus \{\cc\}}$ and $X_{\tau^{\cc}_{\ss \setminus \{\cc\}}}$ are independent.
\end{itemize}

We will now leverage these properties to derive a stochastic representation of the tunneling time $\tcd$. Let $N_q$ be the random variable that counts the number of non-consecutive visits to stable configurations in $\ss \setminus \{\cc\}$ until the configuration $\dd$ is hit, counting as first visit the configuration $\cc$ where we assume the Markov chain starts at time $t=0$. Non-consecutive visits means that we count as actual visit to a stable configuration only the first one after the last visit to a different stable configuration. Property (ii) implies that the random time between these non-consecutive visits does not depend on the last visited stable configuration. In view of property (i), the random variable $N_q$ is geometrically distributed with success probability equal to $(q-1)^{-1}$, \ie
\begin{equation}
\label{eq:Nq}
	\prin{N_q = m} = \left ( 1- \frac{1}{q-1}\right )^{m-1} \frac{1}{q-1}, \quad m \geq 1.
\end{equation}
In particular, note that $N_q$ depends only on $q$ and \textit{not} on the inverse temperature $\b$. The amount of time $\tau^{\cc}_{\ss \setminus \{\cc\}}$ it takes for the Metropolis Markov chain started in $\cc \in \ss$ to hit any stable configuration in $\ss \setminus \{\cc\}$ does not depend on $\cc$, by virtue of property (ii). In view of these considerations and using the independence property (iii), we deduce that for $\cc,\dd \in \ss$, $\cc \neq \dd$
\begin{equation}
\label{eq:geometricrepresentation}
	\tau^{\cc}_{\dd} \ed \sum_{i=1}^{N_q} \tau^{(i)},
\end{equation}
where $\{\tau^{(i)}\}_{i \in \N}$ is a sequence of i.i.d.~random variables distributed as $\tau^{\cc}_{\ss \setminus \{\cc\}}$ and $N_q$ is an independent geometric random variable with success probability $1/(q-1)$ as defined in~\eqref{eq:Nq}. In particular, since both random variables $N_q$ and $\tau^{\cc}_{\ss \setminus \{\cc\}} \ed \tau^{(i)}$ have finite expectation and $\E N_q = q-1$, it immediately follows from Wald's identity that $\E \tau^{\cc}_{\dd}  = (q-1) \cdot \E \tau^{\cc}_{\ss \setminus \{\cc\}}=(q-1) \cdot \E \tau^{(i)}$. Thus, we can rewrite~\eqref{eq:geometricrepresentation} as
\[
	\frac{\tcd}{\E \tcd} \ed \frac{1}{\E N_q} \sum_{i=1}^{N_q} \frac{\tau^{(i)}}{\E \tau^{(i)}}.
\]
Using the fact that $\tau^{(i)} / \E \tau^{(i)} \cd \mathrm{Exp}(1)$ for every $i$ as $\binf$ by virtue of Theorem~\ref{thm:main}(iii), we obtain that $\tcd / \E \tcd$ is asymptotically distrbuted as geometric sum of i.i.d.~unit-mean exponential random variables, which is also exponentially distributed. The resulting exponential distribution of $\tcd /\E \tcd$ has also unit mean, as the geometric sum is scaled by its mean $\E N_q$.\\

\textbf{Remark:} Note that the proof presented in this subsection did not use the fact that $\L$ is a grid graph. Indeed the definition of the family of automorphisms $ \{ \Psi_{k,l}\}_{k,l=1,\dots , q}$ does not depend on the underlying structure of the graph $\L$ and neither the rest of the proof. This means that the statement (iv) in Theorem~\ref{thm:main} would hold for any finite graph $G$, as long as the $q$-state Potts model on such graph $G$ is such that  $\smash{\tau^{\cc}_{\ss \setminus \{\cc\}}/ \E \tau^{\cc}_{\ss \setminus \{\cc\}} \cd \rmexp(1)}$ as $\binf$ for any stable configuration $\cc \in \ss$.

\subsection{Mixing times (Proof of Theorem~\ref{thm:main}(v))}
\label{sub34}
By combining Theorem~\ref{thm:chadw}(i) and (ii), it is easy to check that
\[
	 \max_{\s \neq \cc} \Phi(\s,\cc) -H(\s) = \G(\L) \quad \forall \, \cc \in \ss,
\]
and the statements for both the mixing time and the spectral gap then follow from~\cite[Proposition 3.24]{NZB16}.

\bibliographystyle{plain}
\bibliography{potts_final}
\end{document}